\documentclass[reqno]{amsart}
\usepackage{amsmath,amsfonts,amssymb,amscd,multicol}
\usepackage{xypic}
\usepackage[normalem]{ulem}
\usepackage{verbatim,color}
\usepackage{enumitem}
\setenumerate{label=\textnormal{(\arabic*)}} 
\usepackage{hyperref}
\usepackage{tikz, tikz-cd}
\usepackage[nobysame, non-sorted-cites, alphabetic]{amsrefs}

\newcommand{\on}{\operatorname}

\newcommand{\cat}[1]{\ensuremath{\mathsf{#1}}}

\newcommand{\rcatMod}{\operatorname{Mod-}}

\newcommand{\rGr}{\operatorname{Gr-}\hskip -2pt}
\newcommand{\rGrb}{\operatorname{Gr_+ -}\hskip -2pt}

\newcommand{\wh}{\widehat}

\DeclareMathOperator{\soc}{soc}
\DeclareMathOperator{\rad}{rad}

\newcommand{\mb}{\mathbb}
\newcommand{\N}{\mathbb{N}}
\newcommand{\C}{\mathcal{C}}
\newcommand{\Z}{\mathbb{Z}}
\newcommand{\A}{\mathcal{A}}
\newcommand{\K}{\mathcal{K}}

\newcommand{\LL}{\mathcal{L}}
\newcommand{\OO}{\mathcal{O}}

\newcommand{\beq}{\begin{equation}}
\newcommand{\eeq}{\end{equation}}

\DeclareMathOperator{\id}{id}

\DeclareMathOperator{\Hom}{Hom}

\newcommand{\End}{{\rm End}}

\DeclareMathOperator{\Ext}{Ext}

\newcommand{\wt}{\widetilde}

\DeclareMathOperator{\rgr}{gr-\!} \DeclareMathOperator{\lgr}{\!-gr} \DeclareMathOperator{\rmod}{mod-\!} \DeclareMathOperator{\lmod}{\!-mod}  
\DeclareMathOperator{\lMod}{\!-Mod} 
\DeclareMathOperator{\rMod}{Mod-\!}

\newcommand{\op}{\ensuremath{^\mathrm{op}}}

\newcommand{\separate}{\bigskip}

\newcommand{\coker}{\operatorname{coker}}

\numberwithin{equation}{section}
 \theoremstyle{plain}
\newtheorem{theorem}[equation]{Theorem}

\newtheorem{lemma}[equation]{Lemma}

\newtheorem{corollary}[equation]{Corollary}
\newtheorem*{corollary*}{Corollary}
\newtheorem{proposition}[equation]{Proposition}

\theoremstyle{definition}

\newtheorem{definition}[equation]{Definition}

\newtheorem{remark}[equation]{Remark}

\newtheorem{standing-hypothesis}[equation]{Standing Hypothesis}
\newtheorem{example}[equation]{Example}

\begin{document}

\title{The domain and prime properties for Koszul rings and algebras}

\author{Manuel L. Reyes}
\address{University of California,  Irvine\\ Department of Mathematics\\
340 Rowland Hall\\ Irvine, CA 92697-3875\\ USA}
\email{mreyes57@uci.edu}

\author{Daniel Rogalski}
\address{University of California, San Diego\\ Department of Mathematics\\ 9500 Gilman Dr. \#0112 \\
La Jolla, CA 92093-0112\\ USA}
\email{drogalski@ucsd.edu}

\date{July 17, 2024}

\thanks{This work was supported in part by NSF grant DMS-2201273.  This material is based upon work supported by the National Science Foundation under Grant No.\ DMS-1928930 and by the Alfred P. Sloan Foundation under grant G-2021-16778, while the second author was in residence at the Simons Laufer Mathematical Sciences Institute (formerly MSRI) in Berkeley, California, during the Spring 2024 semester.}

\begin{abstract}
We establish a technique to prove that a Koszul graded ring is prime or a domain using information about its Koszul dual. This is based on a general categorical result that expands on methods of J.\,Y.~Guo, which proves that certain \emph{orbital rings} are prime or domains.
We apply this method to prove that if $A = kQ/I$ is a Koszul twisted Calabi-Yau algebra of dimension~2, such that $Q$ is connected with every vertex having outdegree at least~2, then $A$ is a prime piecewise domain. In particular, the preprojective algebra of a connected quiver whose underlying graph has minimum degree at least~2 is a prime piecewise domain.
\end{abstract}

\maketitle


\section{Introduction}

It is generally expected that a ring with good homological properties should enjoy other desirable ring-theoretic properties. 
Our inspiration for this paper is the following general question: What homological properties are sufficient to ensure that a ring is prime or a domain? 
A famous and exemplary result in this vein is the fact that a commutative noetherian local ring of finite global dimension (i.e., a regular local ring) is a unique factorization domain~\cite{AuslanderBuchsbaum}. 

As one might imagine, it has proved more difficult to deduce such elementwise properties from homological conditions for noncommutative rings and algebras. For instance, it is still an open question whether every noetherian local ring of finite global dimension is a domain. As a result, one must typically invoke further assumptions in order to prove such a structure theorem in the noncommutative case. Some successful examples in the literature are:
\begin{itemize}
    \item \cite{BHM} If $R$ is noetherian and is simple artinian modulo its Jacobson radical, then the quotient of $R$ by its nilradical is isomorphic to a matrix ring over a local domain.
    \item \cite{Ramras, Snider} With the same hypotheses as in the preceding result, if $R$ has global dimension at most $3$, then its nilradical is zero and thus $R$ is a matrix ring over a local domain. 
    \item \cite{Levasseur} If a connected graded noetherian algebra $A$ is Auslander regular, then $A$ is a domain.
    \item \cite{StaffordZhang} If a connected graded ring $R$ is PI and has finite global dimension, then $R$ is a domain. 
    \item \cite{Teo} A fully bounded noetherian ring of finite global dimension that is simple artinian modulo its Jacobson radical is isomorphic to a matrix ring over a local domain.
\end{itemize}
Note that matrix rings over domains are prime rings that satisfy the additional property of being a piecewise domain due to Gordon and Small in~\cite{GordonSmall:piecewise}. We recall the definition of this property in Section~\ref{sec:orbital}.

\separate

Our goal in this work is to provide new homological conditions that, while more subtle than the statements above, are sufficient to guarantee that a ring or algebra is prime or a domain. Notably, we do not place restrictions on the global dimension the ring, nor do we assume that it is noetherian.

We focus on Koszul algebras and the more general Koszul rings of~\cite{BGS},
whose definition we recall in Section~\ref{sec:Koszul modules}.  Such a ring is $\mb{N}$-graded, the degree zero part $S = R_0$ is semisimple, and the Ext algebra $\Ext^\bullet_{R \lMod}(S,S)$ is isomorphic to a quadratic dual ring. We refine a technique due to J.\, Y.~Guo~\cite{Guo:prime} to provide sufficient conditions for a Koszul algebra $R$ to be prime or a domain. (For comments about how our results relate to those in \cite{Guo:prime}, see the end of Section~\ref{sec:application}.)  These criteria are formulated in terms of $\Lambda = \Ext^\bullet_{R \lMod}(S,S)$, and they involve the following condition. If $M$ is a graded right $\Lambda$-module, we let $\Omega^i(M)$ denote its $i$th syzygy module. In case $M$ is Koszul, this syzygy is generated in degree~$i$, so that the graded shift $\Omega^i(M)(i)$ is generated in degree zero.

\begin{definition}
\label{def:Koszulsyzygy}
Let $\Lambda$ be a Koszul ring, denote $S = \Lambda_0$, and fix a decomposition $S = S_1 \oplus \cdots \oplus S_r$ into graded simple right modules. We say that $\Lambda$ satisfies the \emph{Koszul syzygy condition} if, for every integer $i \geq 0$, all $1 \leq j, \ell \leq r$, and every graded $\Lambda$-module homomorphism 
\[
f \colon \Omega^i(S_j)(i) \to S_\ell,
\]
the kernel of $f$ is a Koszul module. This is evidently independent of the decomposition of $S$, and it is equivalent to the condition that every graded maximal submodule of $\Omega^i(S_j)(i)$ is Koszul for all $i \geq 0$.
\end{definition}

The following collects general conditions for a Koszul ring to be a domain, piecewise domain, or prime. Recall that a quiver is \emph{strongly connected} if any two vertices in the quiver are connected by a path.

\begin{theorem}[Theorem~\ref{thm:Koszul and strongly connected}]
Let $R$ be a Koszul ring with $S = R_0$ such that $R_1$ finitely generated as both a left and right $S$-module.  Denote $\Lambda = \Ext^\bullet_{R \lMod}(S,S)$. 
\begin{enumerate}
\item $R$ is a piecewise domain (resp., domain) if and only if  $\Lambda$ satisfies the Koszul syzygy condition (resp., and $S$ is a division ring).
\item Assume that $R$ is a piecewise domain. Then $R$ is prime if and only if the underlying quiver of $R$  (Definition~\ref{def:quivers}) is strongly connected.
\end{enumerate}
\end{theorem}

An immediate consequence of part~(1) above is the following necessary and sufficient condition for a connected graded Koszul algebra to be a domain. 

\begin{theorem}
Let $k$ be an arbitrary field.  Let $A$ be a finitely generated connected graded Koszul $k$-algebra. Then $A$ is a domain if and only if $\Lambda = \Ext^\bullet_A(k,k) \cong A^!$ satisfies the Koszul syzygy condition.
\end{theorem}

\separate

Our original motivation for this investigation was the special case of Artin-Schelter regular algebras~\cite{ArtinSchelter} and their generalizations~\cite{RR:generalized}. 
It was conjectured by Artin, Tate, and Van den Bergh~\cite[p.~338]{ATV2} that every Artin-Schelter regular algebra (of finite GK-dimension) is a domain.
(They further conjectured that these algebras should be noetherian, although that is not the focus of this paper.)
Several results in the literature provide evidence toward the conjecture. In addition to the general results mentioned above, we have:
\begin{itemize}
\item In dimensions~2 and~3, the conjecture is known by classification~\cite{ArtinSchelter, ATV1}. In dimension~2, this is even known without assuming finite GK-dimension~\cite{Zhang}.
\item In dimension~4, the conjecture is proved under the additional assumption that the algebra is noetherian in~\cite[Section~3]{ATV2}.
\end{itemize}
AS regular algebras are famously characterized among connected graded algebras by the condition that their Ext algebras are Frobenius~\cite{Smith,LPWZ}. Thus in order to prove that \emph{Koszul} AS regular algebras are domains, it would suffice by our results to prove that connected graded Koszul Frobenius algebras satisfy the Koszul syzygy condition. While this is still a challenging problem, we hope that it might be achievable through careful study.

On the other hand, it is now well understood that the twisted Calabi-Yau property provides a useful generalization of Artin-Schelter regularity to graded algebras that are not necessarily connected~\cite{RRZ, RR:generalized}.
It seems reasonable to expect that if such an algebra is indecomposable, then it is prime.
We are able to verify this in dimension~2 for many such algebras, whether or not they are noetherian.
The \emph{outdegree} of a vertex $v$ in a quiver is the number of arrows for which $v$ is the source.

\begin{theorem}[Theorem~\ref{thm:quiver-main}]
\label{thm:quiver-main-intro}
Let $Q$ be a quiver in which every vertex has outdegree at least~2. Suppose that $A = kQ/I$ is a graded twisted Calabi-Yau algebra of dimension~2. Then $A$ is a semiprime piecewise domain, and
\[
A \mbox{ is prime} \iff Q \mbox{ is connected} \iff Q \mbox{ is strongly connected}.
\]
\end{theorem}
While some of the algebras appearing in Theorem~\ref{thm:quiver-main-intro} are known to be prime by other methods, we are   unaware of any general result about primeness for the full class of algebras in the theorem above, even for preprojective algebras, though there are many proofs that such algebras are twisted Calabi-Yau.   Moreover, our theorem gives the stronger conclusion that these are piecewise domains, which means the structure of zero-divisors is completely understood.  On the other hand, there are twisted Calabi-Yau algebras $kQ/I$ of dimension $2$ that do not satisfy the hypothesis on $Q$ in the theorem.
These will not be piecewise domains (see Remark~\ref{rem:not-pd}) but are still expected to be prime when $Q$ is connected. 
We suspect that a variation of our method can be used to give restrictions on the possible zero-divisors of these algebras and still prove primeness in some cases, but we do not pursue this here.

A brief outline of the paper is as follows. In Section~\ref{sec:orbital} we introduce orbital rings and prove Theorem~\ref{thm:Guo piecewise}, which gives sufficient conditions for such a ring to be a prime piecewise domain. In Section~\ref{sec:Koszul modules} we show how the Ext algebra of a Koszul ring $R$ can be realized as an orbital ring, in such a way that this criterion can be applied in certain cases to prove that $R$ is a prime piecewise domain. Section~\ref{sec:graded QF} studies the special case of Koszul modules over graded quasi-Frobenius rings. Those results are  applied to the case of twisted CY-2 algebras in Section~\ref{sec:application}. Finally, Appendix~\ref{sec:appendix} explains our conventions on Ext algebras of left versus right modules, both of which appear in the proof of Theorem~\ref{thm:Koszul and strongly connected}.

\subsection*{Acknowledgments}

We are very grateful to Jason Gaddis for introducing us to the paper~\cite{Guo:prime}.  We thank Vladimir Baranovsky, Paul Smith, Toby Stafford, and James Zhang for helpful conversations.

\section{Orbital algebras and the method of Guo}
\label{sec:orbital}

Let $\C$ be an additive category. For an object $M$ of $\C$ and an additive endofunctor
$F \colon \C \to \C$, the associated \emph{orbital ring} is an $\N$-graded ring
\[
O(F,M) = O_\C(F,M) := \bigoplus_{i=0}^\infty \Hom_{\C}(F^i M, M)
\]
whose multiplication is defined, for homogeneous elements $f \in \Hom_{\C}(F^i M,M)$ and $g \in \Hom_{\C}(F^j M,M)$, by
\[
f * g = f \circ F^i(g) \in \Hom_{\C}(F^{i+j} M, M).
\]
If $\C$ is a $k$-linear category over a commutative ring $k$, then $O(F,M)$ naturally inherits the structure of a graded $k$-algebra, and we may refer to it as the \emph{orbital algebra}.
\begin{remark}
There is an obvious ``right-handed'' analogue of the orbital ring
\[
O^r(F,M) := \bigoplus_{i = 0}^\infty \Hom_{\C}(M, F^i M),
\]
whose multiplication is defined on homogeneous elements $f$ of degree $i$ and $g$ of degree $j$ by $f * g = F^j(f) \circ g$. 
Note that if we replace $\C$ and $F$ by the opposite category $\C\op$ and its induced endofunctor $F\op$, we may view this as another instance of the usual orbital ring as follows:
\[
O_{\C\op}(F\op, M) = \bigoplus_{i=0}^\infty \Hom_{\C\op}((F\op)^i M, M) \cong \bigoplus_{i = 0}^\infty \Hom_{\C}(M, F^i M) = O^r(F,M).
\]
Also note that if $F$ is an autoequivalence with quasi-inverse $F^{-1}$, then $O(F,M) \cong O^r(F^{-1},M)$. 
\end{remark}

The following three examples show that a number of ring constructions can be viewed as special cases of orbital rings. These will be examples arising in homological algebra, noncommutative ring theory, and noncommutative projective geometry. 

\begin{example}\label{ex:Ext as orbital}
Let $\A$ be an abelian category with enough injective or projective objects, and let $\C$ denote the derived category $\mathcal{D(A)}$ (either unbounded, or bounded above or below as appropriate). Let $F = \Sigma^{-1}$ be the inverse translation functor of the derived category. Then for any object $M$ of $\A$, the orbital ring takes a form that is well-known~\cite[Corollary~10.7.5]{Weibel} to be isomorphic to the Ext algebra of $M$ with the Yoneda product:
\begin{align*}
O(\Sigma^{-1}, M) &\cong O^r(\Sigma, M) \\ 
&= \bigoplus \Hom_{\mathcal{D}(\A)}(M, \Sigma^i M) \\
&\cong \Ext_\A^\bullet(M,M).
\end{align*}
\end{example}

\begin{example}\label{ex:skew polynomial}
Let $A$ be a ring with an endomorphism $\sigma$. Let $\C = \rcatMod A$ be the category of right $A$-modules.  The twisted bimodule $A^{\sigma}$ (whose left $A$-multiplication is given in the ordinary way and whose right multiplication is given by $m \cdot a := m \sigma(a)$) induces a twisting endofunctor $F = - \otimes_A A^{\sigma}$ of $\C$. Consider the right orbital ring $O^r(F, A)$.  The module $F^i(A) \cong A^{\sigma^i}$ can be identified with $A$ as an abelian group; with such identifications an element of $\Hom_A(A, A^{\sigma^i})$ is of the form $a \sigma^i(-)$ for some $a \in A$, and the action of $F$ on morphisms is the identity.  

Now  for $f = a \sigma^i(-) \in \Hom(A, F^i(A))$ and $g = b \sigma^j(-) \in \Hom(A, F^j(A))$ we have $F^j(f) \circ g = a\sigma^i(b) \sigma^{i+j}(-) \in \Hom_A(A, F^{i+j}(A))$.  On the other hand, in the skew polynomial ring $A[x, \sigma]$, graded with $\deg x = 1$, we have $(ax^i)(bx^j)= a\sigma^i(b) x^{i+j}$.  
We see that 
\begin{align*}
O^r(F,A) &= \bigoplus \Hom_{\C}(A, F^i(A)) \cong A[x; \sigma]
\end{align*}
as graded rings.
\end{example}

\begin{example}\label{ex:twisted homogeneous}
Let $X$ be a scheme with an automorphism $\sigma$ and an invertible sheaf $\LL$. Let $\C = \cat{Qcoh}(X)$, and consider the endofunctor that is defined by pullback via~$\sigma$ composed with tensoring by $\LL$:
\[
F(\mathcal{M}) = \LL \otimes_{\OO_X} \sigma^*\mathcal{M}.
\]
Then the right orbital ring of the structure sheaf $\OO_X$ is isomorphic to the twisted homogeneous coordinate ring~\cite{AV} as follows:
\begin{align*}
O^r(F, \OO_X) &\cong \bigoplus \Hom_{\OO_X}(\OO_X, F^i(\OO_X)) \\
&\cong \bigoplus \Gamma(X, \LL \otimes \sigma^*\LL \otimes \cdots \otimes (\sigma^{i-1})^*\LL) \\
&= B(X, \sigma, \LL).
\end{align*}
\end{example}

The results below will provide sufficient conditions to tame the zero-divisors of an orbital ring $O(F,M)$. In the next theorem, we adapt the method used by Guo in \cite[Theorem~3.2]{Guo:prime} to prove that an orbital ring is prime. In addition to primeness, the method allows one to show 
in cases of interest that an orbital ring is a piecewise domain in the sense of Gordon and Small~\cite{GordonSmall:piecewise}, whose definition we now recall. 

Let $R$ be a ring with a set of pairwise orthogonal idempotents $\{e_i\}$ such that $1 = \sum_{i=1}^r e_i$. Then $R$ is a \emph{piecewise domain (with respect to $\{e_i\}$)} if 
\[
    x \in e_i R e_j, \ y \in e_j R e_\ell, \mbox{ and } xy = 0 \implies x = 0 \mbox{ or } y = 0.
\]
It is easy to see that if $R$ is a piecewise domain with respect to a set of idempotents $\{e_i \}$, then the $e_i$ must be primitive.
A structure theory for piecewise domains was developed in~\cite{GordonSmall:piecewise}, showing that such rings have an ``block upper triangular'' structure whose diagonal blocks are prime piecewise domains, and that prime piecewise domains have a matrix-like structure relative to a list of domains.  The property of being a piecewise domain is logically independent of that of being a prime ring: prime rings with zerodivisors but no nontrivial idempotents cannot be piecewise domains, while triangular matrix algebras such as $\left(\begin{smallmatrix} \mathbb{Q} & \mathbb{Q} \\ 0 & \mathbb{Q} \end{smallmatrix}\right)$ are picewise domains that are not even semiprime. 
On the other hand, we have the following characterization of prime piecewise domains.

\begin{remark}
\label{rem:piecewise prime}
If $R$ is a piecewise domain with respect to $\{e_1, \dots, e_r\}$, then it follows from~\cite[Main Theorem]{GordonSmall:piecewise} that
$R$ is prime if and only if $e_i R e_j \neq 0$ for all $i \neq j$.
\end{remark}

\begin{theorem}\label{thm:Guo piecewise}
Let $\C$ be an 
additive category with an endofunctor $F$. Suppose that $M$ is an object of $\C$ with a decomposition
\[
    M = M_1 \oplus \cdots \oplus M_r
\] 
such that the following hold:
\begin{enumerate}[label = \textnormal{(\roman*)}]
\item every nonzero $f \colon F^d(M_j) \to M_\ell$ is an epimorphism;
\item if $g \colon F^e(M_j) \to M_{\ell}$ is an epimorphism, then so is $F^d(g)$ for any $d \geq 0$.
\end{enumerate}
Then $O(F,M)$ is a piecewise domain with respect to the idempotents in $\End_{\mathcal{C}}(M) = O(F,M)_0$ corresponding to the decomposition of $M$ above.

Furthermore, suppose that \textnormal{(i)--(ii)} hold along with:
\begin{enumerate}[resume, label = \textnormal{(\roman*)}]
\item for $j \neq \ell$, there exists a nonzero morphism $h \colon F^i(M_j) \to M_\ell$ for some $i \geq 0$.
\end{enumerate}
Then $O(F,M)$ is prime.
\end{theorem}

\begin{proof}
Assume~(i) and~(ii) hold, and let $R = O(F,M)$ denote the orbital ring.  For notational simplicity let us write $\C(M,N)$ for $\Hom_{\C}(M,N)$ below.  For each direct summand $M_i$ of $M$, let $e_i \in \C(M,M)$ be the idempotent given by a retraction of $M$ to $M_i$ followed by its corresponding section from $M_i$ into $M$:
\[
    e_i \colon M \overset{\pi_i}{\longrightarrow} M_i \overset{\sigma_i}{\longrightarrow} M.
\]
Then $1 = \sum e_i$ in $R$.  

We claim that we can identify the graded components of the Peirce corners of the orbital ring as 
\begin{equation}\label{eq:corner}
e_i \C(F^d M, M)e_j \cong \C(F^d M_j, M_i).
\end{equation}
Indeed, let $\phi \in \C(F^d M, M)$ be a degree~$d$ element of $R$. Recalling that $e_i$ has degree zero, we have
\begin{align*}
e_i * \phi * e_j &= e_i \circ (\phi * e_j) \\
	&= e_i \circ \phi \circ F^d(e_j).
\end{align*}
Thus we can identify the $d$th graded piece of the corner as
\[
e_i * R_d * e_j = e_i \circ \C(F^d M, M) \circ F^d(e_j). 
\]
Because $F$ is additive, we have $F^d M = \bigoplus_{i=1}^r F^d(M_i)$ with corresponding idempotent decomposition $1 = \sum F^d(e_i)$ in $\C(F^d M, F^d M)$. Because each $\sigma_i$ is a split monomorphism and each $F^d(\pi_j)$ is a split epimorphism, we may cancel them in the second isomorphism below:
\begin{align*}
e_i \circ \C(F^d M, M) \circ F^d(e_j) &= \sigma_i \circ \pi_i \circ \C(F^d M, M) \circ F^d(\sigma^j) \circ F^d(\pi_j) \\
&\cong  \pi_i \circ \C(F^d M, M) \circ F^d(\sigma^j) \\
&=  \C(F^d M_j, M_i).
\end{align*} 
Thus we obtain the identification~\eqref{eq:corner} as claimed.

To prove that the graded ring $R$ is a piecewise domain, it certainly suffices to fix nonzero homogeneous elements $f \in e_i R e_j$ and $g \in e_j R e_\ell$ and prove that $f * g \neq 0$. Letting $d$ and $e$ respectively denote the degrees of $f$ and $g$, under the identification~\eqref{eq:corner} we may view
\begin{align*}
f &\in e_i \C(F^d M, M) e_j = \C(F^d M_j, M_i), \\
g &\in e_j \C(F^e M, M) e_\ell = \C(F^e M_\ell, M_j).
\end{align*}
Hypothesis~(i) implies that $g$ is an epimorphism, so that $F^d(g)$ is also an epimorphism by hypothesis~(ii). Under~\eqref{eq:corner} once more, the product $f * g \in \C(F^{d+e} M, M)$ corresponds to the composite 
\[
    f \circ F^d(g) \in \C(F^{d+e}M_\ell, M_i).
\]
Since $F^d(g)$ is an epimorphism and $f$ is nonzero, the composite above is nonzero. Thus we find $f * g \neq 0$ as desired.

Finally, if~(iii) also holds, then from~\eqref{eq:corner}, we see that for all $j \neq \ell \in \{1,\dots,r\}$ there exists $i \geq 0$ such that $e_j R_i e_\ell \neq 0$. By Remark~\ref{rem:piecewise prime} we deduce that $R$ is prime.
\end{proof}

We state for reference the special case of the preceding result in which there is only one summand.
\begin{corollary}\label{thm:Guo domain}
Let $\C$ be an additive category, let $F$ be an additive endofunctor of $\C$, and let $M$ be a nonzero object of $\C$. Suppose that the following hold:
\begin{enumerate}[label = \textnormal{(\roman*)}]
\item For each $i \geq 0$, every nonzero morphism $f \colon F^i(M) \to M$ is an epimorphism.
\item If $g \colon F^i(M) \to M$ is an epimorphism, then $F^j(g)$ is also an epimorphism for all $j \geq 0$.
\end{enumerate}
Then the orbital ring $O(F,M)$ is a domain. 
\end{corollary}

\separate

Of course, because the ``right-handed'' orbital rings have the form $O^r(F,M) = O_{\C\op}(F\op,M)$, there are dual versions of the results above for such rings. For instance, the dual version of Corollary~\ref{thm:Guo domain} would state that if
\begin{itemize}
\item every nonzero morphism $M \to F^i(M)$ is a monomorphism for $i \geq 0$, and
\item for every monomorphism $g \colon M \to F^i(M)$ and $j \geq 0$, $F^j(g)$ is also a monomorphism,
\end{itemize}
then $O^r(F,M)$ is a domain.
In the context of Examples~\ref{ex:skew polynomial} and~\ref{ex:twisted homogeneous}, we recover the following well-known facts:
\begin{itemize}
\item If $A$ is a domain with an injective endomorphism $\sigma \colon A \to A$, then $A[x; \sigma]$ is also a domain.
\item If $X$ is an integral scheme with an automorphism $\sigma$ and invertible sheaf $\LL$, then the twisted homogeneous coordinate ring $B(X, \sigma, \LL)$ is a domain.
\end{itemize}
These are by no means ``better'' proofs than the traditional ones, but they serve to illustrate that the concept of an orbital algebra covers more constructions than one might expect.

\section{Realizing Ext algebras as orbital algebras}
\label{sec:Koszul modules}

In this section we provide instances of Theorem~\ref{thm:Guo piecewise} that are tailored to the case of 
Ext algebras. While Example~\ref{ex:Ext as orbital} provides one way to view an Ext algebra as an orbital algebra, it is not well suited to the applications of these theorems. Indeed, the hypotheses of these theorems require many morphisms to be epimorphisms, while the only epimorphisms in a triangulated category are the split epimorphisms (cf.~\cite[Exercise~IV.1.1]{GelfandManin}). Thus we seek a different representation of an Ext algebra that works within categories of Koszul modules.

We first recall some basic elements of graded module theory and syzygies. 
Let $\Lambda = \bigoplus_{i=0}^\infty \Lambda_i$ be an $\N$-graded ring.  Let $\rGr \Lambda$ be the category of $\Z$-graded right $\Lambda$-modules, so $\Hom_{\rGr \Lambda}(M,N)$ is the set of degree-preserving module homomorphisms from $M$ to $N$.  Given $M \in \rGr \Lambda$ we write $M = \bigoplus_{n \in \Z} M_n$, where $M_n$ is the $n$th graded piece of $M$.  Given $i \in \Z$, and $M \in \rGr \Lambda$, the shift $M(i)$ is the graded module with $M(i)_n = M_{i+n}$.  
Let $\rgr \Lambda$ be the full subcategory of $\rGr \Lambda$ consisting of finitely generated graded right $\Lambda$-modules.  A module $M \in \rGr \Lambda$ is \emph{left bounded} (or \emph{right bounded}) if $M_n = 0$ for all $n \ll 0$ (respectively $n \gg 0$), and $M$ is \emph{bounded} if it is both left and right bounded.  

For $M, N \in \rGr \Lambda$, we write $\Hom_{\Lambda}(M,N)$ for the usual ungraded Hom-group.  
The following standard result shows that $\Hom_{\Lambda}(M,N)$ obtains a natural $\mb{Z}$-grading in some cases.
\begin{lemma}
\label{lem:easy}
    Let $M, N \in \rGr \Lambda$.  There is a natural inclusion of groups
\[
\bigoplus_{i \in \mb{Z}} \Hom_{\rGr \Lambda}(M, N(i)) \subseteq \Hom_{\Lambda}(M,N)
\]
which is an equality if either (i) $M \in \rgr \Lambda$ or (ii) $M$ is generated by a set of elements contained in finitely many degrees and $N$ is bounded.
\end{lemma}

Our standing assumption on the $\N$-graded ring $\Lambda$ will be that 
\begin{center}
$\Lambda_0 = S$ is a semisimple ring. 
\end{center}
Under this hypothesis, every projective object of $\rGr \Lambda$ is of the form $P = V \otimes_S \Lambda$ for some (necessarily projective) graded right $S$-module $V$. 
Any $M \in \rGr \Lambda$ which is left bounded has a projective cover $(M/J(M)) \otimes_S \Lambda \to M$ in $\rGr \Lambda$, where $J(M)$ is  the graded Jacobson radical of $M$.  In particular, by the graded Nakayama lemma, $M$ is generated in a single degree $n$ if and only if its projective cover $P$ is.

Let $\rGrb \Lambda$ denote the full abelian subcategory of $\rGr \Lambda$ consisting of left bounded modules. 
Since projective covers exist in $\rGrb \Lambda$, every left bounded graded $\Lambda$-module $M$ has a minimal graded projective resolution (see~\cite[Proposition~2.3]{LiWu:Koszul}), and this resolution is unique up to (non-unique) isomorphism of complexes.  A graded module $M$ is \emph{Koszul} if it is generated in degree zero and its minimal graded projective resolution 
\[
\cdots \to P_2 \to P_1 \to P_0 \to M \to 0
\]
has each $P_i$ generated in degree~$i$. We let $\K(\Lambda)$ denote the full subcategory of $\rGrb \Lambda$ consisting of the Koszul modules. 

Given $M \in \rGrb \Lambda$, if $P_M \to M$ is a projective cover in $\rGrb \Lambda$ then the \emph{syzygy} of $M$ is the graded module $\Omega M$ where 
\[
0 \to \Omega M \to P_M \to M \to 0
\] 
is exact. Since projective covers are unique up to isomorphism, the graded module $\Omega M$ is well-defined up to isomorphism. Higher syzygies are defined inductively by $\Omega^n M = \Omega(\Omega^{n-1}M)$, and we set $\Omega^0 M = M$.
We recall a few special properties of syzygy modules. 

\begin{lemma}\label{lem:syzygy}
For a graded ring $\Lambda$ as above, let $M,N \in \rGrb \Lambda$ with $N$ finitely generated semisimple.
\begin{enumerate}
\item $\Ext^i_{\rGr \Lambda}(M,N) \cong \Hom_{\rGr \Lambda}(\Omega^i M,N)$ for all $i \geq 0$. 
\item If $M$ is Koszul then $\Ext^i_{\Lambda}(M,N) \cong \Hom_{\Lambda}(\Omega^i M,N)$ for all $i \geq 0$.
\item The module $M$ is Koszul if and only if $\Omega^n(M)$ is generated in degree~$n$ for all $n \geq 0$.
\end{enumerate}
\end{lemma}

\begin{proof}
(1) Assume $i \geq 1$, as the statement is trivial when $i = 0$.  We get an isomorphism $\Ext^i_{\rGr \Lambda}(M,N) \cong \Ext^1_{\rGr \Lambda}(\Omega^{i-1} M, N)$ inductively from 
the long exact sequence in $\Ext$, see~\cite[Corollary~6.55]{Rotman:homological}.  Thus it suffices to prove that $\Ext^1_{\rGr \Lambda}(M, N) \cong \Hom_{\rGr \Lambda}(\Omega M, N)$.

Beginning with the projective cover 
\[
0 \to \Omega M \to P_M \to M \to 0
\] of $M$, consider the resulting long exact sequence
\[
\begin{tikzcd}[column sep=1em]
    0 \ar[r] & \Hom_{\rGr \Lambda}(M,N) \ar[r] & \Hom_{\rGr \Lambda}(P_M, N) \ar[r, "\alpha"] & \Hom_{\rGr \Lambda}(\Omega M, N) \ar[overlay, dll, out=355,in=175] 
    \\
     & \ar[r] \Ext^1_{\rGr \Lambda}(M,N) \ar[r] & \Ext^1_{\rGr \Lambda}(P_M,N) = 0. & 
\end{tikzcd}
\]
Because $f: P_M \to M$ is a projective cover, $\ker f = \Omega M \subseteq J(P_M)$.  Given $g \colon P_M \to N$ in $\rGr \Lambda$, since $N$ is semisimple we have $J(P_M) \subseteq \ker g$.  Thus $\alpha(g) = g \vert_{\Omega M} = 0$.  Since $\alpha = 0$, 
the connecting morphism now yields an isomorphism $\Hom_{\rGr \Lambda}(\Omega M, N) \cong \Ext^1_{\rGr \Lambda}(M,N)$.

(2) The simple objects in $\rGr \Lambda$ are the shifts of the simple summands of $S = \Lambda_0$, so $N$ is bounded.  By definition, $P_i$ is generated in a single degree for all $i$.  Computing Ext using the minimal graded projective resolution
\[
\cdots \to P_2 \to P_1 \to P_0 \to M \to 0,
\] 
we see that $\Ext^i_{\Lambda}(M, N) \cong \bigoplus_{n \in \mb{Z}} \Ext^i_{\rGr \Lambda}(M, N(n))$ by Lemma~\ref{lem:easy}.  Now apply (1).

(3) In the minimal projective resolution of $M$, $P_i$ is the projective cover of $\Omega^i M$ by construction.  Then $P_i$ is generated in degree $i$ if and only if $\Omega^i M$ is generated in degree $i$.
\end{proof}

Given that $\Omega$ is well-defined on objects, we may attempt to extend it to a functor by defining it on morphisms in the following way. Suppose $f \colon M \to N$ is a morphism in $\rGrb \Lambda$. Fix a lift $\wh{f}$ to the projective covers such that the square on the right commutes:
\begin{equation}\label{eq:Omega defined}
\begin{tikzcd}
0 \ar[r] & \Omega M \ar[d, dashrightarrow, "\Omega(f)"] \ar[r] & P_M \ar[d, "\wh{f}"] \ar[r] & M \ar[d, "f"] \ar[r] & 0 \\
0 \ar[r] & \Omega N \ar[r] & P_N \ar[r] & N \ar[r] & 0 
\end{tikzcd}
\end{equation}
Then we would like to define $\Omega(f)$ as the restriction of $\wh{f}$ to $\Omega M$, whose image necessarily lies in $\Omega N$. However, different choices of the lift $\wh{f}$ may not restrict to the same morphism on the syzygies.  For a simple example, take  $\Lambda = k[x]$ and $M = N = k \oplus k(-1)$, where $k = k[x]/(x)$ is the trivial module. Then $f = \id_M$ is easily seen to have non-unique lifts $\widehat{f}$ restricting to different $\Omega(f)$ in~\eqref{eq:Omega defined}.

The standard fix to the problem above is to work with the projectively stable module category. In this context, the morphism $\Omega(f)$ of~\eqref{eq:Omega defined} is independent of the choice of $\wh{f}$ if it is \emph{viewed as a morphism in the stable category;} see~\cite[\S 5.1.2]{Zimmermann} for details.  In order to apply methods closer to those of~\cite{Guo:prime}, we take a different approach. Namely, we restrict the syzygy operation to a subcategory of graded modules for which the lift $\wh{f}$ is unique, so that $\Omega(f)$ is uniquely determined. This will leave us with a subcategory on which $\Omega$ forms a well-defined functor.

\begin{lemma}
\label{lem:samedegree}
Let $\Lambda$ be a graded ring with $\Lambda_0 = S$ semisimple. Let $M,N \in \rGrb \Lambda$, and let $f \in \Hom_{\rGr \Lambda}(M, N)$. 
If $M$ and $N$ are both generated in degree~$n$, then there is a unique map $\Omega(f)$ completing the commutative diagram~\eqref{eq:Omega defined}.  Furthermore: 
    \begin{enumerate}
        \item If $f$ is injective then so is $\Omega(f)$.
        \item If $f$ is surjective then $\Omega(f)$ is surjective if and only if $\ker f$ is generated in degree $n$.
    \end{enumerate}
\end{lemma}
\begin{proof}
Without loss of generality, by shifting we can assume that $n = 0$.  Then $M_0 \otimes_S \Lambda \to M$ is a projective cover of $M$, and a similar comment holds for $N$.  We have a diagram 
\[
\xymatrix{
0 \ar[r] & \Omega M \ar[d]^{\Omega(f)} \ar[r] & M_0 \otimes_S \Lambda  \ar[d]^{\wh{f}} \ar[r] & M \ar[d]^f \ar[r] & 0 \\
0 \ar[r] & \Omega N \ar[r] & N_0 \otimes_S \Lambda  \ar[r] & N \ar[r] & 0.
}
\]
and the lift $\wh{f}$ which makes the right hand square commute is uniquely determined by $f$, since clearly $\wh{f}_0 = f_0: M_0 \to N_0$, and the projective module $M_0 \otimes_S \Lambda$ is also generated in degree $0$.  Therefore the induced map $\Omega(f)$ on the kernels as in~\eqref{eq:Omega defined} is uniquely determined as well.

By the snake lemma, there is an exact sequence 
\[
0 \to \ker \Omega(f) \to \ker \wh{f} \to \ker f \to \coker \Omega(f) \to \coker\wh{f} \to \coker f \to 0.
\]

(1) If $f$ is injective, then the restriction $f_0: M_0 \to N_0$ is also injective.  It follows that $\wh{f}$ is injective, as $S$ is semisimple and so $- \otimes_S \Lambda$ is exact.  Thus $\ker \wh{f} = 0$.
Then $\ker \Omega(f) = 0$ from the snake lemma and so $\Omega(f)$ is also injective.

(2) If $f$ is surjective, then similarly $f_0: M_0 \to N_0$ and hence $\wh{f}$ is also surjective, so $\coker \wh{f} = 0$.  Then from the snake lemma, $\Omega(f)$ is surjective, that is $\coker \Omega(f) = 0$, if and only if the map $\ker \wh{f} \to \ker f$ is surjective.  Note that 
$\ker \wh{f} \cong V \otimes_S \Lambda$, where $V = \ker f_0 \colon M_0 \to N_0$, so $\ker \wh{f}$ is also generated in degree $0$.  
So the map $\ker \wh{f} \to \ker f$ is surjective if and only if $\ker f$ is generated in degree $0$.
\end{proof}

An easy extension of the argument in Lemma~\ref{lem:samedegree} gives the following, whose proof we leave to the reader.
\begin{lemma}
\label{lem:exactpreserved}
Let $\Lambda$ be an $\mb{N}$-graded ring with $\Lambda_0 = S$ semisimple as above, and let 
\[
\begin{tikzcd}
0 \ar[r] & L \ar[r, "g"] & M \ar[r, "f"] & N \ar[r] & 0
\end{tikzcd}
\] 
be an exact sequence in $\rGr \Lambda$.
If $L$, $M$, and $N$ are all generated in degree $n$, then 
    \[
    \begin{tikzcd}
    0 \ar[r] & \Omega L \ar[r, "\Omega(g)"] & \Omega M \ar[r, "\Omega(f)"] & \Omega N \ar[r] & 0
    \end{tikzcd}
    \] 
is also exact.  
\end{lemma}

We now see how to formally identify an Ext algebra as an orbital algebra.

\begin{proposition}\label{prop:endofunctor}
Let $\Lambda$ be $\mb{N}$-graded with $S = \Lambda_0$ semisimple.   The assignment $F(M) = \Omega M(1)$ determines an additive functor $F \colon \K(\Lambda) \to \K(\Lambda)$ on the category of Koszul right modules. Furthermore, if $M \in \K(\Lambda)$ is a finitely generated semisimple module, then there is an isomorphism of graded rings
\[
O_{\K(\Lambda)}(F,M) \cong \Ext^\bullet_{\Lambda}(M,M).
\]
\end{proposition}

\begin{proof}
The assignment $M \mapsto \Omega M(1)$ sends Koszul modules to Koszul modules thanks to Lemma~\ref{lem:syzygy}(3). Since all Koszul modules are generated in degree~0, Lemma~\ref{lem:samedegree}(1) implies that any morphism $\phi$ in $\K(\Lambda)$ has a unique lift $\widehat{\phi}$ to the projective covers. If $f$ and $g$ are composable morphisms in $\K(\Lambda)$, then $\widehat{g \circ f} = \widehat{g} \circ \widehat{f}$ and it follows that
\[
\Omega(g \circ f) = \Omega(g) \circ \Omega(f).
\]
Similarly, if $f_1$ and $f_2$ are morphisms with the same domain and codomain in $\K(\Lambda)$, we have $\widehat{f_1 + f_2} = \widehat{f_1} + \widehat{f_2}$. This implies
\[
\Omega(f_1 + f_2) = \Omega(f_1) + \Omega(f_2).
\]
Thus $F = \Omega(-)(1)$ is an additive endofunctor of $\K(\Lambda)$.

To see the isomorphism of graded rings, first note that applying Lemma~\ref{lem:syzygy}(2) and Lemma~\ref{lem:easy} yields an isomorphism of graded abelian groups 
\begin{align*}
\Ext^i_{\Lambda}(M, M) &\cong \Hom_{\Lambda}(\Omega^i M, M) 
\cong \bigoplus_{n \in \mb{Z}} \Hom_{\rGr \Lambda}(\Omega^i M, M(n)) \\
&= \Hom_{\rGr \Lambda}(\Omega^i M, M(-i)) = \Hom_{\rGr \Lambda}(F^i M(-i), M(-i)) \\
&\cong \Hom_{\rGr \Lambda}(F^i M, M)  =  \Hom_{\K(\Lambda)}(F^i M,M).
\end{align*}    
Here, all summands of the direct sum in the third term are $0$ except when $n = -i$, because $\Omega^i M$ is generated in degree $i$ and $M$ is entirely contained in degree~$0$.
Summing over $i$ we get an isomorphism $O_{\K(\Lambda)}(F, M) \cong \bigoplus_{i \geq 0} \Ext^i_{\Lambda}(M, M)$ as graded abelian groups.
The proof that this is an algebra isomorphism is the same as in the discussion preceding~\cite[Theorem 3.1]{Guo:prime}.
\end{proof}

Using the description of an Ext algebra as an orbital ring, we can now characterize when such rings are piecewise domains or prime.
\begin{theorem}\label{thm:Koszul prime}
Let $\Lambda$ be a graded ring with $S = \Lambda_0$ semisimple. 
Let $F = \Omega(-)(1)$ be the endofunctor of $\K(\Lambda)$ from Proposition~\ref{prop:endofunctor}.  Suppose $M \in \K(\Lambda)$ with 
\[ 
M = S_1 \oplus S_2 \oplus \dots \oplus S_r,
\]
for some Koszul simple right $\Lambda$-modules $S_i$, and let $e_1,\dots,e_r \in \End_\Lambda(M)$ be the idempotents corresponding to this decomposition.
\begin{enumerate} 
\item The following are equivalent: 
\begin{enumerate}[label = \textnormal{(\alph*)}]
\item $\Ext_{\Lambda}^\bullet(M,M)$ is a piecewise domain with respect to $\{e_i \}$;
\item For every $i \geq 0$ and every nonzero morphism $f \colon F^i(S_{\ell}) \to S_m$, the morphisms $F^j(f)$ are surjective for all $j \geq 0$;
\item For every $i \geq 0$ and every nonzero morphism $f \colon F^i(S_{\ell}) \to S_m$, the module $\ker f$ is Koszul. 
 \end{enumerate}
\item Suppose that the equivalent conditions in~\textnormal{(1)} hold, and assume in addition that 
for any $\ell \neq m$ there exists $i \geq 0$ such that 
\[
    \Hom_{\rGr \Lambda}(F^i(S_{\ell}), S_m) \neq 0.
\]
Then $\Ext_{\Lambda}^\bullet(M,M)$ is prime.
\end{enumerate}
\end{theorem}
\begin{proof}
Let $\K(\Lambda)$ be the additive category of Koszul left $\Lambda$-modules.  By Proposition~\ref{prop:endofunctor}, the Ext algebra $\Ext_\Lambda^\bullet(M,M)$ is isomorphic to the orbital ring  $O_{\K(\Lambda)}(F,M)$, so we can work with the orbital ring instead of the Ext algebra throughout the proof. Let us verify the equivalences in (1):

(b)$\implies$(a): Since $S_m$ is simple, every nonzero morphism $f \colon F^i(S_{\ell}) \to S_m$ is surjective.  Together with~(b) this verifies conditions (i)--(ii) of Theorem~\ref{thm:Guo piecewise}, so $O_{\K(\Lambda)}(F,M)$ is a piecewise domain with respect to the idempotent decomposition of $1 \in \End_\Lambda(M)$ induced by the given expression of $M$ as a direct sum of simple right modules.  

(a)$\implies$(b): 
Fix $0 \neq f \in \Hom_{\rGr \Lambda}(F^i(S_{\ell}),S_m)$, and let $j \geq 0$.  Suppose that $F^j(f) \in \Hom_{\rGr \Lambda}(F^{i+j}(S_{\ell}), F^j(S_m))$ is not surjective.  Since $F^j(S_m)$ is generated in degree $0$ ($S_m$ being Koszul), the image of $F^j(f)$ does not contain all of $(F^j(S_m))_0$.  It follows that 
we can choose a maximal $S$-submodule $P$ of $(F^j(S_m))_0$ so that the $\Lambda$-submodule $N =P + (F^j(S_m))_{\geq 1}$ of $F^j(S_m)$ contains the image of $F^j(f)$.  Now $F^j(S_m)/N$ is a simple module, so it must be isomorphic to some $S_p$.  Let $g: F^j(S_m) \to S_p$ be the corresponding surjection.  Then by construction $g * f = g \circ F^j(f) = 0$.  This contradicts that $O_{\K(\Lambda)}(F,S)$ is a piecewise domain.  So $F^j(f)$ is surjective.

(b)$\implies$(c): Fix $0 \neq f \colon F^i(S_{\ell}) \to S_m$, so necessarily $f$ is surjective.  Suppose according to (b) that $F^j(f)$ is surjective for all $j \geq 0$.   It follows by shifting degree that each $\Omega^j(f)$ is also surjective. In particular, since $\Omega(f)$ is surjective, Lemma~\ref{lem:samedegree} implies that $K:= \ker f$ is generated in degree~0. It follows from Lemma~\ref{lem:exactpreserved} and shifting that applying $F$ yields an exact sequence
\[
\begin{tikzcd}
0 \ar[r] & F(K) \ar[r] & F^{i+1}(S_{\ell}) \ar[r, "F(f)"] & F(S_m) \ar[r] &0,
\end{tikzcd}
\]
where the last two terms are generated in degree $0$.  
Since $F^2(f)$ and hence $\Omega(F(f))$ is surjective, Lemma~\ref{lem:samedegree} again implies that $\ker F(f) = F(K)$ is generated in degree~0.  We may proceed inductively to show that each $F^j(K)$ is generated in degree~$0$. This means that $K = \ker f$ is Koszul, as desired.

(c)$\implies$(b):  This follows from a similar inductive argument using  Lemma~\ref{lem:samedegree} as in the previous paragraph.

Finally, (2) is just a restatement of the last part of Theorem~\ref{thm:Guo piecewise}.
\end{proof}

We close this section by adapting Theorem~\ref{thm:Koszul prime} to formulate criteria under which a Koszul ring is a domain or prime.
Let $R$ be an $\mb{N}$-graded ring with $R_0 = S$ semisimple.  Following \cite{BGS}, we say that $R$ is a \emph{Koszul ring} if the graded left $R$-module $S = \Lambda/\Lambda_{\geq 1}$ is Koszul.  By \cite[Proposition 2.2.1]{BGS}, a ring is Koszul if and only if its opposite ring is.  Thus $R$ is Koszul if and only $S$ is a Koszul right module.
Writing $S = S_1 \oplus \dots \oplus S_r$ for some simple right $\Lambda$-modules $S_i$, note that $S$ is a Koszul module if and only if every $S_i$ is a Koszul module.  

If $R$ is Koszul, then $R$ is generated by $R_1$ as a ring over $S$ \cite[Proposition 2.3.1]{BGS}.  
While \cite{BGS} works in greater generality, we will only consider Koszul rings $R$ such that $R_1$ is finitely generated as an $S$-module on both the left and the right; thus $R_n$ is a finitely generated left and right $S$-module for all $n \geq 0$.
When $R$ is Koszul define 
\[
E(R) = \bigoplus_{i \geq 0} \Ext^i_{R \lMod}(S, S) 
\] 
to be the Ext algebra of the left $R$-module $S$. 
Then $E(R)$ is Koszul with $E(R)_0 \cong S = R_0$, and
\begin{equation}\label{eq:Koszul dual}
R \cong \bigoplus_{i \geq 0}\Ext^i_{\rMod E(R)}(S,S)
\end{equation}
by a right-sided version
of \cite[Theorem 1.2.5]{BGS}, as explained in Appendix~\ref{sec:appendix}.
The ring $E(R)$ is called the \emph{Koszul dual} of $R$.

Our primeness criterion is governed by strong connectedness of a quiver.
The following is not standard but we use it out of convenience, especially because it specializes to the usual underlying quiver for the algebras we consider in Section~\ref{sec:application}. 

\begin{definition}
\label{def:quivers}
Let $R$ be a ring as above.  Fix a decomposition $1 = e_1 + \dots + e_r$ of $1$ as a sum of primitive orthogonal idempotents $e_i \in S$.  We define the \emph{underlying quiver} $Q = Q(R)$ of $R$ as follows.  The vertices in $Q$ are $\{1, \dots, r \}$.  For $1 \leq i, j \leq r$, the number of arrows from $i$ to $j$ in $Q$ is the dimension of $e_j R_1 e_i$ as a right 
$D_i = e_i S e_i = \End_S(e_iS)$-module.  
\end{definition}

It is easy to see that the underlying quiver of $R$ is independent of the choice of primitive idempotent decomposition.  
One could alternatively define $Q(R)$ by looking at the \emph{left} vector space dimension of the corners of $R$. But we are only concerned with strong connectedness of this quiver, and since both such quivers reduce to the same simple directed graph after replacing multiple arrows by a single arrow, their strong connectedness is equivalent.
(It would be arguably more appropriate to consider an underlying species associated to a basic representative of $R$, but we wish to avoid such technicalities here.)

\begin{theorem}
\label{thm:Koszul and strongly connected}
Let $R$ be an $\mb{N}$-graded ring for which $S = R_0 = R/R_{\geq 1}$ is semisimple and $R_1$ is finitely generated as a left and right $S$-module.  Suppose that $R$ is Koszul and let $\Lambda = E(R)$ be the Koszul dual ring.  Then 
\begin{enumerate}
\item $R$ is a piecewise domain (with respect to any primitive idempotent decomposition $1 = e_1 + e_2 + \dots + e_r$ for $e_i \in S$) if and only if $\Lambda$ satisfies the Koszul syzygy condition (Definition~\ref{def:Koszulsyzygy}).
\item Let $R$ be a piecewise domain as in (1).  Then $R$ is prime if and only if the underlying quiver $Q(R)$ is strongly connected.
\item $R$ is a domain if and only if $S$ is a division ring and $\Lambda$ satisfies the Koszul syzygy condition.
\end{enumerate}
\end{theorem}
\begin{proof}
(1) The ring $\Lambda$ is also Koszul, so it satisfies the hypothesis of Theorem~\ref{thm:Koszul prime}.  By the equivalence of (1a) and (1c) in that result, $\Ext_{\rMod \Lambda}^\bullet(S,S)$ is a piecewise domain if and only if for all $1 \leq j, \ell \leq r$ and $i \geq 0$, every graded $\Lambda$-module homomorphism $f \colon \Omega^i(S_j)(i) \to S_\ell$ has Koszul kernel.  This is precisely the Koszul syzygy condition.  Now by~\eqref{eq:Koszul dual} we have $R \cong \Ext_{\rMod \Lambda}^\bullet(S,S)$.  Note that condition (1c) in Theorem~\ref{thm:Koszul prime} depends only on the set of isomorphism classes of simple modules, which is independent of how $S$ is expressed as a direct sum of simple modules.  Thus the piecewise domain property is independent of the primitive idempotent decomposition of $1$ in $\End_S(S)$. The 

(2)  By Remark~\ref{rem:piecewise prime}, $R$ is prime if and only if given distinct indices $i \neq j$ we have  $e_j R e_i \neq 0$.  Since $R$ is Koszul, $R$ is generated as a ring by $R_1$ over $S$.  Thus an element of $e_j R_n e_i$ is a sum of products of the 
form $x_n x_{n-1} \dots x_1$, where $x_{\alpha} \in f_{\alpha+1} R_1 f_{\alpha}$ for some sequence of idempotents $f_0, f_1, \dots, f_{n+1} \in \{e_1, \dots, e_r\}$ such that $f_0 = e_i$ and $f_{n+1} = e_j$.  The definition of $Q$ implies that 
there exists at least one arrow from $i$ to $j$ in $Q$ if and only if $e_j R_1 e_i \neq 0$.  If $Q$ is strongly connected, then for any $i, j$ there is a path from $i$ to $j$ of some length $n$, which produces such a sequence of idempotents $f_{\alpha}$ and nonzero elements $0 \neq x_{\alpha} \in f_{\alpha+1} R_1 f_{\alpha}$.  Then $0 \neq x_nx_{n-1} \dots x_1 \in e_j R e_i$
by the piecewise domain property.  Conversely, if $Q$ is not strongly connected because there is no path $i \to j$, then the description above shows that $e_j R_n e_i = 0$ for all $n$, so that $e_j R e_i = 0$.

(3) If $R$ is a domain, then $S = R_0$ is a semisimple domain---hence a division ring---and by part~(1) $\Lambda$ satisfies the Koszul syzygy condition. Conversely, if $S$ is a division ring and $\Lambda$ is satisfies the Koszul syzygy condition, then $R$ is a piecewise domain with respect to $\{1\}$, which means that $R$ is a domain.
\end{proof}

\begin{remark}
    \label{rem:other side}
Definition~\ref{def:Koszulsyzygy} might properly be called the \emph{right} Koszul syzygy condition for $\Lambda$, as it depends on properties of the syzygies of the simple right modules.  In the context of Theorem~\ref{thm:Koszul and strongly connected}, we see that this is equivalent to $E(\Lambda)$ being a piecewise domain, which is a right/left symmetric condition.  It follows that for Koszul rings $\Lambda$ the left and right Koszul syzygy conditions are equivalent, which is why we simply called this the Koszul syzygy condition.
\end{remark}

\section{Graded quasi-Frobenius rings}
\label{sec:graded QF}


By Theorem~\ref{thm:Koszul and strongly connected}, in order to prove a Koszul ring $R$ is a piecewise domain (or prime), we want to understand 
when its dual $\Lambda = E(R)$ satisfies the Koszul syzygy condition.   We now investigate this condition for the special class of quasi-Frobenius rings $\Lambda$.

Let $\Lambda$ be an $\N$-graded ring, with $\Lambda_0 = S$ semisimple.  We are interested in the case that $\Lambda$ is artinian.  Thus there exists $d \geq 0$ such that $\Lambda_i = 0$ for $i > d$, whence
\begin{equation}\label{eq:graded length}
\Lambda = \Lambda_0 \oplus \Lambda_1 \oplus \cdots \oplus \Lambda_d
\end{equation}
Choose $d$ minimal such that $\Lambda_d \neq 0$. In this case, the \emph{graded length} of $\Lambda$ is equal to~$d+1$. 

We say that $\Lambda$ (graded and artinian, as above) is \emph{quasi-Frobenius} if it is self-injective as a right $\Lambda$-module (among many equivalent characterizations~\cite[\S 15]{Lam:Lectures}).  It is equivalent to require $\Lambda$ to be injective in the category of all modules or in the category of graded modules \cite[Proposition~3.4]{DNN:Frobenius}.  

\begin{remark}\label{rem:qF}
We recall some structural facts about graded quasi-Frobenius rings.  Let $1 = e_1 + \cdots + e_r$ be a sum of primitive orthogonal idempotents in $\Lambda_0 = S$, and denote $P_i = e_i \Lambda$ and $S_i = e_i S \cong P_i / \rad(P_i)$.  
As shown, for example in~\cite[Section~4]{DNN:Frobenius}, we have the following:
\begin{itemize}
\item The graded projective and graded injective $\Lambda$-modules coincide.
\item The modules $P_i$ represent all graded projective (resp., injective) indecomposable right $\Lambda$-modules up to isomorphism.  
\item $\soc(P_i)$ is simple, and every graded simple occurs (up to shift) as the socle of some $P_i$. 
\item The functors $\Hom_{\rmod \Lambda}(-,\Lambda)$ and $\Hom_{\Lambda \lmod}(-,\Lambda)$ are mutually quasi-inverse, forming a duality between finitely generated graded right and left $\Lambda$-modules.
\end{itemize}
\end{remark}

We are primarily interested here in graded quasi-Frobenius algebras $\Lambda$ such that $\soc(\Lambda_{\Lambda}) = \Lambda_d$, which forces all of the indecomposable projectives  $P_i = e_i \Lambda$ to have socle in degree $d$ as well.  Let $M \in \Lambda \lgr$.  Suppose that $M$ is generated in degree $0$, so that 
its projective cover has the form $P_M = \oplus_{i=1}^m \Lambda e_{n_i} \to M$.  Applying the duality functor $\Hom_{\Lambda}(-, \Lambda)$, we see that $M^{\vee}: = \Hom_{\Lambda}(M, \Lambda) \to \oplus_{i=1}^m e_{n_i} \Lambda$ must be an injective hull in $\rgr \Lambda$.  In particular, since each $e_{n_i} \Lambda$ has socle in degree $d$, $M^{\vee}$ has socle entirely in degree $d$.  Conversely, by applying the inverse functor, if $M^{\vee}$ has socle entirely in degree $d$ then $M$ is generated in degree $0$.  By shifting we see that a left module $M$ is generated in a single degree $n$ if and only if its dual right module $M^{\vee}$ has socle in a single degree $d-n$.

Now a number of results from Section~\ref{sec:Koszul modules} about syzygies dualize immediately to results about cosyzygies, as follows.  For example, by dualizing the construction of 
$\Omega$, we see that for any $M \in \rgr \Lambda$ there is an exact sequence 
\[
0 \to M \to E(M) \to \Omega^{-1}(M) \to 0
\]
where $E(M)$ is the graded injective hull of $M$.  In particular $\soc(M) = \soc E(M)$.  This exact sequence is determined up to isomorphism by $M$, and the \emph{cosyzygy module} $\Omega^{-1}(M)$ is well-defined up to isomorphism. 

The next result follows immediately from applying the duality $\Hom_{\Lambda}(-, \Lambda)$ to the (left module versions of) Lemmas~\ref{lem:syzygy} and Lemma~\ref{lem:exactpreserved} for finitely generated modules, using that having a socle in a single degree dualizes to being generated in a single degree.  
\begin{lemma}
\label{lem:cosyzygy}
Let $\Lambda$ be graded quasi-Frobenius with $\Lambda_0 = S$ semisimple and $\soc(\Lambda) = \Lambda_d$. 
\begin{enumerate}
\item If $M, N \in \rgr \Lambda$ both have socles contained in the same degree~$n$, then there is a uniquely determined morphism $\Omega^{-1}(f) \in \Hom_{\rGr \Lambda }(\Omega^{-1} M, \Omega^{-1} N)$. Furthermore:
\begin{enumerate}[label = \textnormal{(\alph*)}]
\item If $f$ is surjective then so is $\Omega^{-1}(f)$.
\item If $f$ is injective, then $\Omega^{-1}(f)$ is injective if and only if $\coker f$ has socle contained in degree~$n$.
\end{enumerate}
\item If $0 \to L \overset{g}{\to} M \overset{f}{\to} N \to 0$ is a short exact sequence in $\rgr \Lambda$ and all of $L$, $M$, and $N$ have socles in degree~$n$, then
\[
\begin{tikzcd}
0 \ar[r] & \Omega^{-1} L \ar[r, "\Omega^{-1} g"] & \Omega^{-1} M \ar[r, " \Omega^{-1} f"] &  \Omega^{-1} N \ar[r] & 0
\end{tikzcd}
\]
is also exact.
\end{enumerate}
\end{lemma}

For any object $M$ in $\rGr \Lambda$ we can define the objects $FM = (\Omega M)(1)$ and $F^{-1}M = (\Omega^{-1}M)(-1)$, with the caveat that $F$ and $F^{-1}$ act only on objects in general; given a morphism $f: M \to N$ then we consider $F(f)$ only if $M$ and $N$ are generated in the same degree, and $F^{-1}(f)$ only if $M$ and $N$ have socle in the same degree.  The suggestive notation $F^{-1}$ comes from the fact that if one passes to the stable category (where one mods out by morphisms that factor through projective modules), then $F$ and $F^{-1}$ become actual quasi-inverse functors.  Consequently, as long as one avoids projective modules these are inverse operations on objects:
\begin{lemma} 
\label{lem:F-props}
Let $\Lambda$ be a graded quasi-Frobenius ring with $\Lambda_0 = S$ semisimple and $\soc(\Lambda) = \Lambda_d$.  Let $X(\Lambda)$ be the set of objects in $\rGr \Lambda$ without projective summands (equivalently, without injective summands).
\begin{enumerate}
    \item  If $M \in X(\Lambda)$ we have $F(F^{-1}(M)) \cong M$ and $F^{-1}(F(M)) \cong M$.
    \item If $M \in X(\Lambda)$ is indecomposable then so are $F(M)$ and $F^{-1}(M)$. 
\end{enumerate}
\end{lemma}
\begin{proof}
(1) Let $M \in X(\Lambda)$ and take the exact sequence $0 \to \Omega M \overset{f}{\to} P_M \overset{g}{\to} M \to 0$ for a graded projective cover $g$. Then $P_M$ is also injective, so there is a submodule $E \subseteq P_M$ such that $\Omega M\to E$ is a graded injective hull of $\Omega M$.  Since $E$ is injective we have $P_M = E \oplus Q$ for some complement $Q$, and 
$g \vert_Q: Q \to M$ is an injection because $f(\Omega M) = \ker(g) \subseteq E$.  Then $g(Q)$ is a projective-injective subobject of $M$ and hence a projective summand of $M$.  This forces $Q = 0$ and so $f: \Omega M \to P_M$ is an injective hull.  It follows that $\Omega^{-1}(\Omega M) \cong M$ and putting in the shifts, $F^{-1}(F M) \cong M$.  A dual argument shows that $\Omega(\Omega^{-1} M) \cong M$ and $F(F^{-1} M) \cong M$.

(2) This is an easy consequence of (1).
\end{proof}

The following class of modules is convenient to focus in our further analysis.
\begin{definition}
Let $\Lambda = \Lambda_0 \oplus \cdots \oplus \Lambda_d$ be an $\mb{N}$-graded ring of graded length~$d+1$. A module $M \in \rGr \Lambda$ is \emph{compressed} if its graded decomposition is of the form
\[
M = M_0 \oplus \cdots \oplus M_{d-1},
\]
or equivalently, $M$ is positively graded with $M_i = 0$ for $i \geq d$.
\end{definition}

\begin{lemma}
\label{lem:F-props2}
Let $\Lambda$ be a graded quasi-Frobenius ring with $\Lambda_0 = S$ semisimple and $\soc(\Lambda) = \Lambda_d$.  Let $M \in \rGr \Lambda$ be compressed.  
\begin{enumerate}
\item If $M$ is generated in degree~0, then $F(M)$ is compressed and has socle in degree $d-1$.  
\item If $M$ has socle in degree $d-1$, then $F^{-1}(M)$ is compressed and is generated in degree $0$.  
    \end{enumerate}
\end{lemma}
\begin{proof}
(1) Consider the projective cover
\[
0 \to \Omega M \to P \to M \to 0.
\]
Since $M$ is generated in degree zero, $P$ is also generated in degree zero, so $P$ is a direct sum of indecomposable projectives of the form $e_i \Lambda$.  Thus $P$ is concentrated in degrees $0$ through $d$.  Furthermore, the surjection of $P$ onto $M$ is an isomorphism in degree zero by minimality. It follows that $K = \Omega M$ is concentrated in degrees~$1$ through~$d$: $K = K_1 \oplus \cdots \oplus K_d$.  Thus $F(M) = K(1)$ is compressed.  By the assumption $\soc(\Lambda) = \Lambda_d$, every indecomposable projective $e_i \Lambda$ has socle in degree $d$, so $\soc(P) = P_d$.  Then the submodule $K$ must have $\soc(K) = K_d$.  Thus $F(M) = K(1)$ has socle in degree $d-1$.

(2) This follows by a dual argument to (1).
\end{proof}

For the remainder of this section, we will specialize to the case where $\Lambda$ has graded length~$3$ and is Koszul, for which compressed modules have some additional special properties.  First, most indecomposable compressed modules must be generated in degree $0$ and have socle in degree $1$.
\begin{lemma}
\label{lem:compressed-options}
Suppose that $\Lambda = \Lambda_0 \oplus \Lambda_1 \oplus \Lambda_2$ has graded length~3, where $\Lambda_0 = S$ is semisimple and $\soc(\Lambda) = \Lambda_2$. Suppose that $M = M_0 \oplus M_1 \in \rgr \Lambda$ is an indecomposable compressed module. 
Then exactly one of the following holds:
\begin{enumerate}[label=\textnormal{(\roman*)}]
\item $M \cong S_i$ for a simple $\Lambda_0$-module $S_i$;
\item $M \cong S_i(-1)$ for a simple $\Lambda_0$-module $S_i$;
\item $M$ is generated in degree~$0$ and has socle in degree~$1$.
\end{enumerate}
\end{lemma}

\begin{proof}
Suppose that $M$ does not fall under either of cases~(i) or~(ii). If $M$ is not generated in degree~0, then $N =  M_0 \Lambda$ is a proper submodule of $M$, so that $N_1 \subsetneq M_1$. Taking a $\Lambda_0$-module decomposition 
\[
M_1 = N_1 \oplus V, 
\]
we see that $V$ is in fact a $\Lambda$-submodule of $M$. It follows that $M = N \oplus V$ as $\Lambda$-modules, contradicting indecomposability of $M$. 

Similarly, assume toward a contradiction that $\soc(M) \supsetneq M_1$. Then we may decompose 
\[
M_0 = \soc(M)_0 \oplus W
\]
as $\Lambda_0$-modules. This gives a $\Lambda$-module decomposition $M = \soc(M) \oplus W$, which again contradicts indecomposability.
Thus $M$ satisfies~(iii).
\end{proof}

Finally, we show that indecomposable non-Koszul modules and factor modules of the $F^n(S_j)$ have a very special form.
\begin{lemma}
\label{lem:notkoszul}
\label{lem:special-maps}
Let $\Lambda$ be a graded quasi-Frobenius Koszul ring with $\Lambda_0 = S$ semisimple and $
\soc(\Lambda) = \Lambda_2$.  Let $S = S_1 \oplus S_2 \oplus \dots \oplus S_r$ with each $S_i$ simple.
\begin{enumerate}
    \item If $M \in \rgr \Lambda$ is compressed, indecomposable and not Koszul, then $M \cong F^{-m}(S_i(-1))$ for some $m \geq 0$ and some $i$.
    \item Suppose that $0 \neq M \in \rgr \Lambda$ is indecomposable and that there is an epimorphism $f: F^n(S_j) \to M$ in $\rgr \Lambda$ for some $n \geq 0$.  Then $M \cong F^m(S_i)$ for some $i$ and $0 \leq m \leq n$. 
\end{enumerate}
\end{lemma}
\begin{proof}
(1) Let $m \geq 0$ be minimal such that $\Omega^m M$ is not generated in degree $m$.  Equivalently, $m$ is minimal such that $F^m M$ is not generated in degree $0$.  By Lemma~\ref{lem:F-props}(2), $F^m M$ is indecomposable, and since each $F^i M$ with $0 \leq i < m$ is generated in degree $0$, applying Lemma~\ref{lem:F-props2} inductively we see that $F^m M$ is still compressed.  The only possibility according to Lemma~\ref{lem:compressed-options} is $F^m M\cong S_i(-1)$ for some $i$.  Then $M \cong F^{-m}(S_i(-1))$ for some $i$, by Lemma~\ref{lem:F-props}(1).

(2) If $F^n(S_j)$ is simple, then $F^n(S_j) \cong M$ and the result is obvious.  So assume not; in particular, assume $n > 0$.  We proceed by induction on~$n$.   We know that $F^n(S_j)$ is compressed, indecomposable, and generated in degree $0$ by Lemmas~\ref{lem:F-props} and \ref{lem:F-props2} and the fact that $\Lambda$ is Koszul.  Since $F^n(S_j)$ is not simple, according to the options in Lemma~\ref{lem:compressed-options}, $F^n(S_j)$ is both generated in degree $0$ and has socle in degree~$1$.  Similarly, $M$ is generated in degree $0$, as a epimorphic image of $F^n(S_j)$.  If $M \cong S_i$ for some $i$, we are done by taking $m = 0$.  Otherwise, by Lemma~\ref{lem:compressed-options}, $M$ also has socle in degree $1$.
In this case, since both $F^n(S_j)$ and $M$ have socle in degree $1$, we can apply $F^{-1}$ to the morphism $f$ by Lemma~\ref{lem:cosyzygy}(1), obtaining $F^{-1}(f): F^{n-1}(S_j) \to F^{-1}(M)$, which is still surjective by Lemma~\ref{lem:cosyzygy}(1).  Also, $F^{-1}(M)$ is still indecomposable by Lemma~\ref{lem:F-props}.   By induction we conclude that $F^{-1}(M) \cong F^p(S_i)$ for some $0 \leq p \leq n-1$ and some $i$.  Applying $F$, we have $M \cong F^{p+1}(S_i)$.
\end{proof}

\section{Application to twisted Calabi-Yau algebras}
\label{sec:application}

In this final section we give the main application of our results to prove a certain class of algebras are prime piecewise domains.  

For simplicity we restrict to the setting of an $\N$-graded algebra $A$ over an arbitrary field $k$ such that $S = A_0 = k^r$ is a product of copies of the base field $k$.  We assume that $A$ is a finitely generated $k$-algebra.  Let $1 = e_1 + \dots + e_r$ be the decomposition of $1$ as a sum of primitive orthogonal idempotents.  

The underlying quiver $Q = Q(A)$, as defined in Definition~\ref{def:quivers}, has vertices labeled $\{1, 2, \dots, r \}$ and $\dim_k e_j A_1 e_i$ arrows from $i$ to $j$.  It is well known that there is an isomorphism of graded algebras $A \cong kQ/I$, where the path algebra $kQ$ is graded by path length with right to left composition, and where $I$ is a homogeneous ideal with $I \subseteq kQ_{\geq 2}$.  
Let $B$ be the incidence matrix of $Q$, so $B \in M_r(\mb{Z})$ with $B_{ji} = \dim_k e_j A_1 e_i$ equal to the number of arrows in $Q$ from $i$ to $j$.   For a finitely generated $\mb{Z}$-graded $A$-module $M$, we define the Hilbert series of $M$ as $h_M(t) = \sum (\dim_k M_n) t^n$.  

Suppose that $A$ satisfies all of the properties above and is also Koszul.  Then it is known that $A$ is quadratic, that is $A \cong kQ/I$ with $I = (I_2)$ generated by quadratic relations.  In addition, the Koszul dual $E(A)$ is isomorphic to $A^! = kQ^{op}/(I_2^{\perp})$, where $Q^{op}$ is the opposite quiver and $I_2^{\perp}$ is the orthogonal complement to $I_2$ in a natural sense; see \cite[Theorem 4.1]{MV:selfinjective} for details.

Let us see now how to verify the Koszul syzygy condition for Koszul quasi-Frobenius rings of graded length 3. The \emph{indegree} of a vertex $v$ in a quiver is the number of arrows for which $v$ is their target.

\begin{proposition}
\label{prop:quiver-hs}
    Let $\Lambda = kQ/I$ be a finite-dimensional quasi-Frobenius Koszul algebra with $\soc(\Lambda) = \Lambda_2$.  Let $Q$ have vertex set $\{1, \dots, r \}$ and write $1 = e_1 + \dots + e_r$ where $e_i$ is the trivial path at vertex $i$.  Let $S_i = e_i \Lambda$.  Let $B$ be the incidence matrix of $Q$.  Suppose that every vertex of $Q$ has indegree at least~2. Then:
    \begin{enumerate}
    \item $h_{e_j \Lambda}(t) = 1 + d_j t + t^2$ where $d_j = \sum_i B_{ji} \geq 2$.
    \item For all $n \geq 0$ and all $j$, there exist $a > b$ with $h_{F^n(S_j)}(t) = a + bt$.
    \item For all $n \geq 0$ and all $j$, there exist $a < b$ with $h_{F^{-n}(S_j(-1))}(t) = a + bt$.
    \item $\Lambda$ satisfies the Koszul syzygy condition.
    \end{enumerate}
\end{proposition}
\begin{proof}
(1) $(e_j \Lambda)_1 = \bigoplus_i e_j \Lambda_1 e_i$ is equal to the $k$-span of all arrows from other vertices $i$ to $j$.  So this has dimension $\sum_i B_{ji}$.  Since the indecomposable module $e_i \Lambda$ has simple top and simple socle (as a consequence of Remark~\ref{rem:qF}), we get $h_{\Lambda e_i}(t) = 1 + (\sum_i B_{ji})t + t^2$, as all simple modules are $1$-dimensional over $k$. Now the assumption that every vertex $j$ has indegree at least two implies that $d_j = \sum_i B_{ji} \geq 2$ for all $j$.   

(2) By Lemma~\ref{lem:F-props2}, since $S_j$ is a Koszul module we have $F^n(S_j)$ is compressed for all $n \geq 0$ and so 
$h_{F^n(S_j)}(t) = a_n + b_n t$ for some $a_n, b_n \geq 0$.  We show $a_n > b_n$ for all $n \geq 0$ by induction, the base case being obvious.   Assume $a_n > b_n$ for some $n$.  Since $F^n(S_j)$ has a top consisting of a direct sum of $a_n$ simple modules, if $P \to F^n(S_j)$ is a projective cover then $P$ is a direct sum of $a_n$ indecomposable projectives, so $h_P(t) = a_n + dt + a_n t^2$ with $d \geq 2a_n$ by (1).  Then 
\[
h_{\Omega(F^n(S_j))}(t) = h_P(t) - h_{F^n(S_j)}(t) = (d-b_n)t + a_nt^2, 
\]
and so $h_{F^{n+1}(S_j)}(t) = (d-b_n) + a_n t$.  It follows that 
\[
a_{n+1} = d-b_n \geq 2a_n -b_n = a_n + (a_n -b_n) > a_n = b_{n+1}
\]
as desired.

(3) This is proved dually to (2), since the indecomposable injective modules are also the modules $e_i \Lambda$.

(4) Supppose that the Koszul syzygy condition fails, so there is an exact sequence of the form 
\[
0 \to K \to F^n(S_j) \to S_i \to 0
\]
in which $K$ is not Koszul.  Since $F^n(S_j)$ is compressed, so is $K$.  Consider an indecomposable decomposition $K = K_1 \oplus \dots \oplus K_p$.  At least one summand is not Koszul, and without loss of generality we can take it to be $K_1$.  let $N = K_2 \oplus \dots \oplus K_p$.  Thinking of $N$ as a submodule of $F^n(S_j)$, we now have an exact sequence
\begin{equation}
\label{eq:forcontradiction}
 0 \to K_1 \to F^n(S_j)/N \to S_i \to 0
\end{equation}
where $K_1$ is compressed, indecomposable and not Koszul.  By Lemma~\ref{lem:notkoszul}(1), $K_1 \cong F^{-q}(S_{\ell}(-1))$ for some $q \geq 0$ and $\ell$.  In particular, $h_{K_1}(t) = c + dt$ with $c < d$, by (3).

On the other hand, write $L : = F^n(S_j)/N = L_1 \oplus \dots \oplus L_t$ for some indecomposable modules $L_i$.  Each $L_i$ is isomorphic to some $F^m(S_k)$ with $m \geq 0$ by Lemma~\ref{lem:special-maps}(2).  Thus $H_{L_i}(t) = a_i + b_i t$ with $a_i > b_i$ by (2), and hence $h_L(t) = a + bt$ with $a > b$.

From \eqref{eq:forcontradiction} we obtain 
the relation $1 + c + dt = a + bt$, which leads to the contradiction $a = 1 + c \leq d = b$.
\end{proof}

While the results in Section~\ref{sec:graded QF} only depended on the quasi-Frobenius property, the cases of our main interest satisfy a stronger property.  Following~\cite{DNN:Frobenius}, a finite-dimensional $\mb{N}$-graded $k$-algebra $\Lambda$ is called a \emph{$d$-graded Frobenius algebra} if there is an isomorphism of graded right (equivalently, left) $\Lambda$-modules
\[
\Lambda^* = \Hom_k(\Lambda, k) \cong \Lambda(d).
\]
A $d$-graded Frobenius algebra $\Lambda$ is quasi-Frobenius of graded length $d+1$ and with $\soc(\Lambda) = d$.  The converse is not true in general, 
but does hold in the following special case.
\begin{proposition}\label{prop:path algebra Frobenius}
Suppose that $\Lambda = kQ/I$ is graded self-injective of graded length $d+1$. If $Q$ is connected, then $\Lambda$ is $d$-graded Frobenius and $Q$ is strongly connected.
\end{proposition}
\begin{proof}
Under the hypotheses, it is proved in~\cite[Theorem~3.3]{MV:selfinjective} that the radical series and the socle series of $\Lambda$ coincide. The radical filtration also coincides with the graded filtration, from which it follows that $\soc(\Lambda) = \Lambda_d$.   Now $\Lambda/J(\Lambda) = S$ is isomorphic to a direct sum of $r$ nonisomorphic simple modules, where $r$ is the number of vertices in $Q$.  Each simple module $S_i$ embeds in $\soc\Lambda(d)$ by Remark~\ref{rem:qF},  and each indecomposable summand of $\Lambda$ has simple socle, so that $\soc \Lambda$ has length~$r$.  It follows that $S \cong \soc \Lambda(d)$.
By \cite[Proposition~5.4]{DNN:Frobenius}, $\Lambda$ is $d$-graded Frobenius. It follows from~\cite[Corollary~3.4]{Green:Frobenius} that $Q$ is strongly connected.
\end{proof}

We can now prove our main results.

\begin{theorem}
\label{thm:quiver-main}
Let $A = kQ/I$ for a homogeneous ideal such that every vertex in $Q$ has outdegree at least~2.  Suppose that $A$ is Koszul and twisted Calabi-Yau of dimension $2$.  Then:
\begin{enumerate}
    \item $A$ is a product of prime piecewise domains. In particular, $A$ is a semiprime piecewise domain.
    \item $A$ is prime if and only if $Q$ is connected, if and only if $Q$ is strongly connected.
\end{enumerate}
\end{theorem}

\begin{proof}
If $Q$ is not connected, we may partition it into connected components $Q_j$ to obtain $kQ = kQ_1 \oplus \cdots \oplus kQ_p$. Note that each $Q_i$ still satisfies the condition that every vertex is the source of at least two arrows. As in~\cite[Lemma~2.7]{RR:generalized}, this leads to a decomposition $A = A_1 \oplus \cdots \oplus A_p$ for indecomposable algebras $A_j \cong kQ_j/I_j$. Each $A_j$ is again twisted Calabi-Yau of dimension $2$, by~\cite[Proposition~4.6]{RR:generalized}. Certainly $A$ is certainly not prime if $p \geq 2$.  Thus we may replace $Q$ by a connected component, and it suffices to prove that in this case $A$ is a prime piecewise domain and $Q$ is strongly connected. 

Let $S = A_0$, and denote $\Lambda = \Ext_{A \lMod}^\bullet(S,S)$.  Since $\on{gl.dim}(A) = 2$, $\Lambda = \Lambda_0 \oplus \Lambda_1 \oplus \Lambda_2$ has graded length $3$.  The algebra $A$ is also generalized AS regular by \cite[Theorem 1.5]{RR:generalized}, in the sense defined there.  It follows from~\cite[Theorem~1.5]{LiWu:Koszul} that $\Lambda \cong kQ^{op}/(I_2^{\perp})$ is self-injective. As $Q\op$ is also connected, by Proposition~\ref{prop:path algebra Frobenius} $\Lambda$ is $2$-graded Frobenius and $Q\op$ is strongly connected, so $Q$ is also strongly connected.  In particular, $\soc(\Lambda) = \Lambda_2$ and since every vertex in $Q$ is the source of at least two arrows, every vertex in $Q\op$ is the target of at least two arrows.  Now by Proposition~\ref{prop:quiver-hs}(4), $\Lambda$ satisfies the Koszul syzygy condition.  So $A$ is a prime piecewise domain by Theorem~\ref{thm:Koszul and strongly connected}.
\end{proof}

Some important concrete instances of CY-2 algebras are the preprojective algebras. The result above specializes as follows.
The \emph{degree} of a vertex $v$ in a quiver $\Gamma$ means its degree in the underlying graph of $\Gamma$, which is equal to the sum of the indegree and outdegree of $v$.

\begin{corollary}\label{cor:preprojective}
If $\Gamma$ is a (connected) quiver in which every vertex has degree at least~2, then the preprojective algebra $\Pi(\Gamma)$ is a (prime) piecewise domain.
\end{corollary}

\begin{proof}
Preprojective algebras are Koszul rings by~\cite[Corollary~4.3]{BrennerButlerKing}. It is also well-known that they are 2-Calabi-Yau. (For instance, this can be deduced by a Hilbert series argument, combining~\cite{EtingofEu, Minamoto} and~\cite[Lemma~7.6]{RR:growth}.) 
The assumption on $\Gamma$ ensures that every vertex of the double quiver $Q = \overline{\Gamma}$ is the source (and target) of at least two arrows. Since the preprojective algebra has a graded presentation of the form $\Pi(\Gamma) \cong k\overline{\Gamma}/I$, the result now follows from Theorem~\ref{thm:quiver-main}.
\end{proof}

We close with some examples and remarks about our main results.

\begin{remark}
If $kQ/I$ is twisted Calabi-Yau of dimension $2$, then the incidence matrix $B$ of $Q$ satisfies $B = P B^T$ for some permutation matrix $P$ \cite[Proposition 7.1]{RR:growth}.  From this one can see that the hypothesis in Theorem~\ref{thm:quiver-main} that every vertex of $Q$ has outdegree at least~2 is equivalent to the requirement that every vertex has indegree at least~2.  
\end{remark}

\begin{example}
The algebra $A = k \langle x, y \rangle/(xy)$ is the standard example of a connected graded Koszul algebra of global dimension $2$ which is not AS regular.   The Koszul dual $E(A) = A^!$ is non-Frobenius and does not satisfy the Koszul syzygy condition, as can be checked directly.  Of course $A \cong \Ext^\bullet_{E(A)}(k,k)$ is obviously not a domain.  Similar examples exist in dimension $3$.  Thus if the methods presented here are to be applied to prove that a Koszul AS~regular algebra is a domain, we expect the Frobenius property of the Ext algebra to play an essential role in deducing the Koszul syzygy condition.
\end{example}

\begin{remark}
\label{rem:not-pd}
Suppose that $A = kQ/I$ is Koszul twisted Calabi-Yau of dimension $2$, but $Q$ has a vertex which is not the source (equivalently, target) of two arrows.  Then in fact $A$ is not a piecewise domain. To see this, let $\Lambda = A^! \cong kQ^{\op}/(I_2^{\perp})$. By Proposition~\ref{prop:quiver-hs}, some indecomposable projective $e_j \Lambda$ has Hilbert series $1 + t + t^2$.  This means that $F(S_j)$ has Hilbert series $1 + t$, and so there is a simple object $S_i$ and a surjective map $F(S_j) \to S_i$ whose kernel is isomorphic to $S_{\ell}(-1)$ for some $\ell$.  Thus the Koszul syzygy condition fails for $\Lambda$ and 
hence $A = \Lambda^!$ is not a piecewise domain by Theorem~\ref{thm:Koszul and strongly connected}; in fact, some product of two arrows is a relation. (For an explicit example of such an algebra, one can take $A$ to be the preprojective algebra on a quiver of type $\widetilde{D}_n$.)
\end{remark}

The preceding example shows that Theorem~\ref{thm:quiver-main} is the best possible result about the piecewise domain property for Koszul twisted CY algebras of dimension $2$.  On the other hand, we conjecture that all indecomposable twisted CY algebras are prime.  Primeness for some twisted CY-2 algebras and preprojective algebras has been established by other methods in some cases:

\begin{example} 
The class of tame orders studied in~\cite{RVdB} includes the case of PI twisted Calabi-Yau algebras $kQ/I$ of quadratic growth, although they were not described in that language.  These are shown to be prime in \cite[Proposition 2.15]{RVdB}.  For such examples the incidence matrix of $Q$ has spectral radius exactly $2$ (see \cite[Theorem 7.8]{RR:growth}).

For a more explicit example, for any extended Dynkin quiver $\Gamma$, the preprojective algebra $k\overline{\Gamma}/I$ is Calabi-Yau, and is known to be prime by these earlier results.  On the other hand, $Q = \overline{\Gamma}$ has two arrows leaving each vertex only if $\Gamma$ is of type $\wt{A}_n$, but not for the $\wt{D}$- and $\wt{E}$-types. Such preprojective algebras were proved to be prime even earlier in~\cite[Lemma~6.1]{BGL:preprojective}.  It would be interesting to investigate the more subtle zero-divisor structure of these algebras using the methods of our paper, and see if it leads to an alternative proof of primeness.
\end{example}

Finally, we discuss how the specific results in~\cite{Guo:prime} are related to our work.
The result~\cite[Theorem 3.2]{Guo:prime} gives conditions under which an orbital algebra is prime; the method we use in Theorem~\ref{thm:Guo piecewise} is modeled after this result, although Guo does not consider the piecewise domain property.  Guo uses his result to claim that all self-injective Koszul algebras $\Lambda$ satisfying an additional condition on the underlying quiver have prime Ext algebras \cite[Theorem 4.1]{Guo:prime}, and as a consequence that all Artin-Schelter regular algebras are prime \cite[Theorem 4.2]{Guo:prime}.   The proof of \cite[Theorem 4.1]{Guo:prime} does not appear to be correct, as it seems to assume that one can treat the syzygy operation $\Omega$ as a functor on the entire category $\mathcal{C}$ of $\Lambda$-modules without projective summands, and that as such it is an equivalence of categories and so preserves epimorphisms.   As we discussed in Section~\ref{sec:Koszul modules}, $\Omega$ may not have a well-defined action on morphisms between modules when they are not generated in the same single degree.  Moreover, Lemma~\ref{lem:samedegree} shows that the condition under which an epimorphism is preserved by $\Omega$ is subtle.

In fact, if a noetherian AS regular algebra is prime, then it must also be a domain.  This folklore result follows easily from a graded version of the argument in \cite[Theorem 3.2]{BHM}, which was given for local rings.  So if \cite[Theorem 4.2]{Guo:prime} were correct, it would have proved that noetherian Koszul AS regular algebras are domains.

We believe that the conjecture that AS regular algebras are domains is still completely open. Unfortunately, we do not yet see how to use the methods of this paper even to reproduce the known result from \cite[Section 4]{ATV2} that the conjecture holds for noetherian AS regular algebras of dimension 3 and 4.  Furthermore, to our knowledge the literature contains little information about  the piecewise domain or prime condition for the more general class of Koszul twisted Calabi-Yau algebras $kQ/I$, even in  dimension 3.

\appendix 

\section{Side conventions and Ext algebras}
\label{sec:appendix}

The proof of Theorem~\ref{thm:Koszul and strongly connected} involves a passage between left and right Ext algebras. While this paper is written in such a way as to minimize such technicalities, some choices must be made in order to make the statement precise. This short appendix explains our conventions on Ext algebras of left modules and how they interface with those of~\cite{BGS}.

We follow the custom that homomorphisms of modules over a ring $R$ act from the opposite side of the $R$-scalars. Thus, morphisms of left modules act on the right of the module and are composed from left to right. 
One pleasant consequence of this is that the endomorphism ring of the free left module ${}_R R^n$, viewed as a space of row vectors, is the matrix ring $\mathbb{M}_n(R)$ acting from the right by ordinary matrix multiplication. 
In general, composition takes the explicit form
\begin{align*}
\Hom({}_R L,{}_R M) \times \Hom({}_R M,{}_R P) &\to \Hom({}_R L,{}_R P), \\
(f, g) &\mapsto fg,
\end{align*}
where if $f$ and $g$ act by $x \mapsto (x)f$ and $y \mapsto (y)g$, then $fg$ acts by $x \mapsto ((x)f)g$.
On the level of categories, this means when identifying left $R$-modules with right $R\op$-modules, we must also reverse the order of composition so that
\[
R \lMod \cong (\rMod R\op)\op.
\]
Despite appearances, this has quite natural consequences. For instance, consider the $R$-dual functors 
$\Hom_{R \lMod}(-,R)$ and $\Hom_{\rMod R}(-,R)$ acting on the full subcategories of finitely generated free modules. On objects, they interchange free left modules of row vectors with free right modules of column vectors. But for morphisms, they act identically on the spaces of $n \times m$ matrices over $R$, preserving the order of matrix multiplication. 

This practice has subtle consequences for the Ext algebra of a left module.
The usual homological constructions apply without alteration for the category $\rMod A$ since we still compose from the left. But for the left module category, the ``double opposite'' convention above means that we have a ring isomorphism
\[
\Ext^\bullet({}_A M, {}_A M) = \Ext^\bullet(M_{A\op}, M_{A\op})\op.
\]
While this initially seems disconcerting and perhaps unnatural, it can be explained in the following way. Let $0 \leftarrow M \leftarrow P^\bullet$ be a projective resolution of $M$ by left $A$-modules, where the differentials of $P^\bullet$ naturally act from the right. Then $P^\bullet$ is a left DG $A$-module. Our conventions mean that it is also a \emph{right} DG module over the DG ring $E = \Hom({}_A P^\bullet, {}_A P^\bullet)$ (using the internal hom of DG modules, as in~\cite[\S 3.4]{Yekutieli}). Thus $P^\bullet$ is actually a DG $(A,E)$-bimodule, and our Ext conventions simply mean that
\[
\Ext^\bullet({}_A M, {}_A M) \cong H^\bullet(E).
\]

To see one concrete benefit of this convention, let $A$ be a connected graded Koszul algebra with quadratic dual $A^!$. These conventions allow us to have the left module and right module versions of the Yoneda algebra isomorphic to the same quadratic dual:
\[
\Ext^\bullet({}_A k, {}_A k) \cong A^! \cong \Ext^\bullet(k_A, k_A).
\]
Under the ordinary conventions, the Ext algebra of ${}_A k$ would be the \emph{opposite} of the quadratic dual and thus the opposite of the Ext algebra of $k_A$.

Finally, suppose that $R$ is a graded ring with $R_0 = S$ semisimple. We can view $S = R/R_{\geq 1}$ as a trivial left or right module, and form two Ext algebras $E_l(R) = \Ext^\bullet_{R \lMod}(S,S)$ and $E_r(R) = \Ext^\bullet_{\rMod R}(S,S)$ according to the conventions above. 
Assume as in Section~\ref{sec:Koszul modules} that $R$ is a graded Koszul ring for which $V := R_1$ is finitely generated as both a left and a right $S$-module, so that $R$ is \emph{left} and \emph{right finite} in the sense of~\cite{BGS}. Then $R$ is a quadratic ring, with a presentation of the form $R \cong T_S(V)/(L)$ for an $(S,S)$-subbimodule of relations $L \subseteq V \otimes_S V$.
As in~\cite[Section~2]{BGS}, we can use the left and right $S$-dual functors $(-)^\vee := \Hom_{S \lMod}(-,S)$ and  ${}^\vee (-) := \Hom_{\rMod S}(-,S)$ along with orthogonal bimodules of relations $L^\perp \subseteq V^\vee \otimes_S V^\vee \cong (V \otimes_S V)^\vee$ and ${}^\perp L \subseteq {}^\vee V \otimes_S {}^\vee V \cong {}^\vee (V \otimes_S V)$ to define the left and right \emph{quadratic dual} rings $R^! := T_S(V^\vee)/(L^\perp)$ and ${}^! R := T_S({}^\vee V)/({}^\perp L)$,
which satisfy $R \cong {}^!(R^!) \cong ({}^! R)^!$.
Then translating~\cite[Theorem~2.10.1]{BGS} into the current conventions, we get $E_l(R) \cong R^!$ and $E_r(R) \cong {}^! R$.
Collecting these observations, we arrive at 
\[
E_r(E_l(R)) \cong {}^!(R^!) \cong R \quad \mbox{and} \quad E_l(E_r(R)) \cong ({}^! R)^! \cong R,
\]
with the first isomorphism being~\eqref{eq:Koszul dual}.

\bibliography{koszul-domain}

\begin{bibdiv}
\begin{biblist}

\bib{AuslanderBuchsbaum}{article}{
      author={Auslander, Maurice},
      author={Buchsbaum, D.~A.},
       title={Unique factorization in regular local rings},
        date={1959},
        ISSN={0027-8424},
     journal={Proc. Nat. Acad. Sci. U.S.A.},
      volume={45},
       pages={733\ndash 734},
         url={https://doi.org/10.1073/pnas.45.5.733},
}

\bib{ArtinSchelter}{article}{
      author={Artin, Michael},
      author={Schelter, William~F.},
       title={Graded algebras of global dimension {$3$}},
        date={1987},
     journal={Adv. in Math.},
      volume={66},
      number={2},
       pages={171\ndash 216},
         url={https://doi.org/10.1016/0001-8708(87)90034-X},
}

\bib{ATV1}{incollection}{
      author={Artin, M.},
      author={Tate, J.},
      author={Van~den Bergh, M.},
       title={Some algebras associated to automorphisms of elliptic curves},
        date={1990},
   booktitle={The {G}rothendieck {F}estschrift, {V}ol. {I}},
      series={Progr. Math.},
      volume={86},
   publisher={Birkh\"{a}user Boston, Boston, MA},
       pages={33\ndash 85},
}

\bib{ATV2}{article}{
      author={Artin, M.},
      author={Tate, J.},
      author={Van~den Bergh, M.},
       title={Modules over regular algebras of dimension {$3$}},
        date={1991},
     journal={Invent. Math.},
      volume={106},
      number={2},
       pages={335\ndash 388},
         url={https://doi.org/10.1007/BF01243916},
}

\bib{AV}{article}{
      author={Artin, M.},
      author={Van~den Bergh, M.},
       title={Twisted homogeneous coordinate rings},
        date={1990},
     journal={J. Algebra},
      volume={133},
      number={2},
       pages={249\ndash 271},
         url={https://doi.org/10.1016/0021-8693(90)90269-T},
}

\bib{BrennerButlerKing}{article}{
      author={Brenner, Sheila},
      author={Butler, Michael C.~R.},
      author={King, Alastair~D.},
       title={Periodic algebras which are almost {K}oszul},
        date={2002},
     journal={Algebr. Represent. Theory},
      volume={5},
      number={4},
       pages={331\ndash 367},
         url={https://doi.org/10.1023/A:1020146502185},
}

\bib{BGL:preprojective}{article}{
      author={Baer, Dagmar},
      author={Geigle, Werner},
      author={Lenzing, Helmut},
       title={The preprojective algebra of a tame hereditary {A}rtin algebra},
        date={1987},
     journal={Comm. Algebra},
      volume={15},
      number={1-2},
       pages={425\ndash 457},
         url={https://doi.org/10.1080/00927878708823425},
}

\bib{BGS}{article}{
      author={Beilinson, Alexander},
      author={Ginzburg, Victor},
      author={Soergel, Wolfgang},
       title={Koszul duality patterns in representation theory},
        date={1996},
     journal={J. Amer. Math. Soc.},
      volume={9},
      number={2},
       pages={473\ndash 527},
         url={https://doi.org/10.1090/S0894-0347-96-00192-0},
}

\bib{BHM}{article}{
      author={Brown, K.~A.},
      author={Hajarnavis, C.~R.},
      author={MacEacharn, A.~B.},
       title={Noetherian rings of finite global dimension},
        date={1982},
     journal={Proc. London Math. Soc. (3)},
      volume={44},
      number={2},
       pages={349\ndash 371},
         url={https://doi.org/10.1112/plms/s3-44.2.349},
}

\bib{DNN:Frobenius}{article}{
      author={D\u{a}sc\u{a}lescu, S.},
      author={N\u{a}st\u{a}sescu, C.},
      author={N\u{a}st\u{a}sescu, L.},
       title={Graded (quasi-){F}robenius rings},
        date={2023},
     journal={J. Algebra},
      volume={620},
       pages={392\ndash 424},
         url={https://doi.org/10.1016/j.jalgebra.2023.01.003},
}

\bib{EtingofEu}{article}{
      author={Etingof, Pavel},
      author={Eu, Ching-Hwa},
       title={Koszulity and the {H}ilbert series of preprojective algebras},
        date={2007},
     journal={Math. Res. Lett.},
      volume={14},
      number={4},
       pages={589\ndash 596},
         url={https://doi.org/10.4310/MRL.2007.v14.n4.a4},
}

\bib{GelfandManin}{book}{
      author={Gelfand, Sergei~I.},
      author={Manin, Yuri~I.},
       title={Methods of homological algebra},
     edition={Second},
      series={Springer Monographs in Mathematics},
   publisher={Springer-Verlag, Berlin},
        date={2003},
        ISBN={3-540-43583-2},
         url={https://doi.org/10.1007/978-3-662-12492-5},
}

\bib{Green:Frobenius}{article}{
      author={Green, Edward~L.},
       title={Frobenius algebras and their quivers},
        date={1978},
     journal={Canadian J. Math.},
      volume={30},
      number={5},
       pages={1029\ndash 1044},
         url={https://doi.org/10.4153/CJM-1978-087-5},
}

\bib{GordonSmall:piecewise}{article}{
      author={Gordon, Robert},
      author={Small, L.~W.},
       title={Piecewise domains},
        date={1972},
     journal={J. Algebra},
      volume={23},
       pages={553\ndash 564},
         url={https://doi.org/10.1016/0021-8693(72)90121-4},
}

\bib{Guo:prime}{incollection}{
      author={Guo, Jin~Yun},
       title={On the primeness of an {A}rtin-{S}chelter regular {K}oszul
  algebra},
        date={2005},
   booktitle={Representations of algebras and related topics},
      series={Fields Inst. Commun.},
      volume={45},
   publisher={Amer. Math. Soc., Providence, RI},
       pages={135\ndash 140},
}

\bib{Lam:Lectures}{book}{
      author={Lam, T.~Y.},
       title={Lectures on {M}odules and {R}ings},
      series={Graduate Texts in Mathematics},
   publisher={Springer-Verlag, New York},
        date={1999},
      volume={189},
         url={https://doi.org/10.1007/978-1-4612-0525-8},
}

\bib{Levasseur}{article}{
      author={Levasseur, Thierry},
       title={Some properties of noncommutative regular graded rings},
        date={1992},
     journal={Glasgow Math. J.},
      volume={34},
      number={3},
       pages={277\ndash 300},
         url={https://doi.org/10.1017/S0017089500008843},
}

\bib{LPWZ}{article}{
      author={Lu, Di~Ming},
      author={Palmieri, John~H.},
      author={Wu, Quan~Shui},
      author={Zhang, James~J.},
       title={Koszul equivalences in {$A_\infty$}-algebras},
        date={2008},
        ISSN={1076-9803},
     journal={New York J. Math.},
      volume={14},
       pages={325\ndash 378},
         url={http://nyjm.albany.edu:8000/j/2008/14_325.html},
}

\bib{LiWu:Koszul}{article}{
      author={Li, Haonan},
      author={Wu, Quanshui},
       title={Generalized {K}oszul algebra and {K}oszul duality},
        date={2023},
     journal={J. Algebra},
      volume={619},
       pages={643\ndash 685},
         url={https://doi.org/10.1016/j.jalgebra.2022.12.023},
}

\bib{Minamoto}{unpublished}{
      author={Minamoto, Hiroyuki},
       title={The {H}ilbert series of the preprojective algebras},
        date={2024},
        note={\href{https://arxiv.org/abs/2402.15732}{\texttt{arXiv:2402.15732
  [math.RA]}}},
}

\bib{MV:selfinjective}{article}{
      author={Mart\'{\i}nez-Villa, Roberto},
       title={Graded, selfinjective, and {K}oszul algebras},
        date={1999},
     journal={J. Algebra},
      volume={215},
      number={1},
       pages={34\ndash 72},
         url={https://doi.org/10.1006/jabr.1998.7728},
}

\bib{Ramras}{article}{
      author={Ramras, Mark},
       title={Orders with finite global dimension},
        date={1974},
     journal={Pacific J. Math.},
      volume={50},
       pages={583\ndash 587},
         url={http://projecteuclid.org/euclid.pjm/1102913244},
}

\bib{Rotman:homological}{book}{
      author={Rotman, Joseph~J.},
       title={An introduction to homological algebra},
     edition={Second},
      series={Universitext},
   publisher={Springer, New York},
        date={2009},
         url={https://doi.org/10.1007/b98977},
}

\bib{RR:growth}{article}{
      author={Reyes, Manuel~L.},
      author={Rogalski, Daniel},
       title={Growth of graded twisted {C}alabi-{Y}au algebras},
        date={2019},
     journal={J. Algebra},
      volume={539},
       pages={201\ndash 259},
         url={https://doi.org/10.1016/j.jalgebra.2019.07.029},
}

\bib{RR:generalized}{article}{
      author={Reyes, Manuel~L.},
      author={Rogalski, Daniel},
       title={Graded twisted {C}alabi-{Y}au algebras are generalized
  {A}rtin-{S}chelter regular},
        date={2022},
     journal={Nagoya Math. J.},
      volume={245},
       pages={100\ndash 153},
         url={https://doi.org/10.1017/nmj.2020.32},
}

\bib{RRZ}{article}{
      author={Reyes, Manuel},
      author={Rogalski, Daniel},
      author={Zhang, James~J.},
       title={Skew {C}alabi-{Y}au algebras and homological identities},
        date={2014},
     journal={Adv. Math.},
      volume={264},
       pages={308\ndash 354},
         url={https://doi.org/10.1016/j.aim.2014.07.010},
}

\bib{RVdB}{article}{
      author={Reiten, Idun},
      author={Van~den Bergh, Michel},
       title={Two-dimensional tame and maximal orders of finite representation
  type},
        date={1989},
     journal={Mem. Amer. Math. Soc.},
      volume={80},
      number={408},
         url={https://doi.org/10.1090/memo/0408},
}

\bib{Smith}{incollection}{
      author={Smith, S.~Paul},
       title={Some finite-dimensional algebras related to elliptic curves},
        date={1996},
   booktitle={Representation theory of algebras and related topics ({M}exico
  {C}ity, 1994)},
      series={CMS Conf. Proc.},
      volume={19},
   publisher={Amer. Math. Soc., Providence, RI},
       pages={315\ndash 348},
}

\bib{Snider}{article}{
      author={Snider, Robert~L.},
       title={Noncommutative regular local rings of dimension 3~{II}},
        date={1991},
        ISSN={0092-7872,1532-4125},
     journal={Comm. Algebra},
      volume={19},
      number={9},
       pages={2495\ndash 2496},
         url={https://doi.org/10.1080/00927879108824274},
}

\bib{StaffordZhang}{article}{
      author={Stafford, J.~T.},
      author={Zhang, J.~J.},
       title={Homological properties of (graded) {N}oetherian {${\rm PI}$}
  rings},
        date={1994},
     journal={J. Algebra},
      volume={168},
      number={3},
       pages={988\ndash 1026},
         url={https://doi.org/10.1006/jabr.1994.1267},
}

\bib{Teo}{article}{
      author={Teo, Kok-Ming},
       title={Homological properties of fully bounded {N}oetherian rings},
        date={1997},
     journal={J. London Math. Soc. (2)},
      volume={55},
      number={1},
       pages={37\ndash 54},
         url={https://doi.org/10.1112/S0024610796004620},
}

\bib{Weibel}{book}{
      author={Weibel, Charles~A.},
       title={An {I}ntroduction to {H}omological {A}lgebra},
      series={Cambridge Studies in Advanced Mathematics},
   publisher={Cambridge University Press, Cambridge},
        date={1994},
      volume={38},
         url={https://doi.org/10.1017/CBO9781139644136},
}

\bib{Yekutieli}{book}{
      author={Yekutieli, Amnon},
       title={Derived categories},
      series={Cambridge Studies in Advanced Mathematics},
   publisher={Cambridge University Press, Cambridge},
        date={2020},
      volume={183},
}

\bib{Zhang}{article}{
      author={Zhang, James~J.},
       title={Non-{N}oetherian regular rings of dimension {$2$}},
        date={1998},
     journal={Proc. Amer. Math. Soc.},
      volume={126},
      number={6},
       pages={1645\ndash 1653},
         url={https://doi.org/10.1090/S0002-9939-98-04480-3},
}

\bib{Zimmermann}{book}{
      author={Zimmermann, Alexander},
       title={Representation {T}heory: {A} {H}omological {A}lgebra {P}oint of
  {V}iew},
      series={Algebra and Applications},
   publisher={Springer, Cham},
        date={2014},
      volume={19},
        ISBN={978-3-319-07967-7; 978-3-319-07968-4},
         url={https://doi.org/10.1007/978-3-319-07968-4},
}

\end{biblist}
\end{bibdiv}

\end{document}